\documentclass[reqno]{amsart}%
\usepackage{amssymb,mathtools,anysize}
\marginsize{2.5cm}{2.5cm}{2.5cm}{2.5cm}
\usepackage{ifpdf}
\ifpdf
 \usepackage[hyperindex]{hyperref}
\else
 \expandafter\ifx\csname dvipdfm\endcsname\relax
 \usepackage[hypertex,hyperindex]{hyperref}%
 \else
 \usepackage[dvipdfm,hyperindex]{hyperref}%
 \fi
\fi

\numberwithin{equation}{section}
\allowdisplaybreaks[4]
\theoremstyle{plain}
\newtheorem{thm}{Theorem}[section]
\newtheorem{cor}{Corollary}[section]
\newtheorem{lem}{Lemma}[section]
\theoremstyle{remark}
\newtheorem{rem}{Remark}[section]
\DeclareMathOperator{\td}{d}
\DeclareMathOperator{\ti}{i}
\DeclareMathOperator{\te}{e}
\DeclareMathOperator{\arccosh}{arccosh}
\DeclareMathOperator{\arcsinh}{arcsinh}
\DeclareMathOperator{\arctanh}{arctanh}
\DeclareMathOperator{\bell}{B}

\begin{document}

\title[Series expansions for powers of inverse cosine]
{Taylor's series expansions for real powers of functions containing squares of inverse (hyperbolic) cosine functions, explicit formulas for special partial Bell polynomials, and series representations for powers of circular constant}

\author[F. Qi]{Feng Qi}
\address{Institute of Mathematics, Henan Polytechnic University, Jiaozuo 454010, Henan, China\newline\indent
School of Mathematical Sciences, Tiangong University, Tianjin 300387, China}
\email{\href{mailto: F. Qi <qifeng618@gmail.com>}{qifeng618@gmail.com}, \href{mailto: F. Qi <qifeng618@hotmail.com>}{qifeng618@hotmail.com}, \href{mailto: F. Qi <qifeng618@qq.com>}{qifeng618@qq.com}}
\urladdr{\url{https://qifeng618.wordpress.com}, \url{https://orcid.org/0000-0001-6239-2968}}

\dedicatory{Dedicated to people facing and battling COVID-19}

\begin{abstract}
In the paper, by virtue of expansions of two finite products of finitely many square sums, with the aid of series expansions of composite functions of (hyperbolic) sine and cosine functions with inverse sine and cosine functions, and in the light of properties of partial Bell polynomials, the author establishes Taylor's series expansions of real powers of two functions containing squares of inverse (hyperbolic) cosine functions in terms of the Stirling numbers of the first kind, presents an explicit formula of specific partial Bell polynomials at a sequence of derivatives of a function containing the square of inverse cosine function, derives several combinatorial identities involving the Stirling numbers of the first kind, demonstrates several series representations of the circular constant Pi and its real powers, recovers series expansions of positive integer powers of inverse (hyperbolic) sine functions in terms of the Stirling numbers of the first kind, and also deduces other useful, meaningful, and significant conclusions.
\end{abstract}



\thanks{This paper was typeset using\AmS-\LaTeX}

\maketitle
\tableofcontents

\section{Simple preliminaries}

In this paper, we use the notation
\begin{equation*}
\mathbb{N}=\{1,2,\dotsc\},\quad \mathbb{N}_0=\{0,1,2,\dotsc\}, \quad \mathbb{N}_-=\{-1,-2,\dotsc\}, \quad \mathbb{Z}=\{0,\pm1,\pm2,\dotsc\}.
\end{equation*}
\par
The classical Euler gamma function $\Gamma(z)$ can be defined~\cite[Chapter~3]{Temme-96-book} by
\begin{equation*}
\Gamma(z)=\lim_{n\to\infty}\frac{n!n^z}{\prod_{k=0}^n(z+k)}, \quad z\in\mathbb{C}\setminus\{0,-1,-2,\dotsc\}.
\end{equation*}
The rising factorial, or say, the Pochhammer symbol, of $\beta\in\mathbb{C}$ is defined~\cite[p.~7497]{AIMS-Math20210491.tex} by
\begin{equation}\label{rising-Factorial}
(\beta)_n=\prod_{k=0}^{n-1}(\beta+k)
=
\begin{cases}
\beta(\beta+1)\dotsm(\beta+n-1), & n\in\mathbb{N};\\
1, & n=0.
\end{cases}
\end{equation}
For $\alpha,\beta\in\mathbb{C}$ with $\alpha+\beta\in\mathbb{C}\setminus\{0,-1,-2,\dotsc\}$, extended Pochhammer symbol $(\beta)_\alpha$ is defined~\cite{Hansen-B-1975} by
\begin{equation}\label{extended-Pochhammer-dfn}
(\beta)_{\alpha}=\frac{\Gamma(\alpha+\beta)}{\Gamma(\beta)}.
\end{equation}
The falling factorial for $\beta\in\mathbb{C}$ and $n\in\mathbb{N}$ is defined by
\begin{equation*}
\langle\beta\rangle_n=\prod_{k=0}^{n-1}(\beta-k)
=
\begin{cases}
\beta(\beta-1)\dotsm(\beta-n+1), & n\in\mathbb{N};\\
1, & n=0.
\end{cases}
\end{equation*}
The extended binomial coefficient $\binom{z}{w}$ is defined~\cite{DIGCBC-Wei.tex} by
\begin{equation*}
\binom{z}{w}=
\begin{dcases}
\frac{\Gamma(z+1)}{\Gamma(w+1)\Gamma(z-w+1)}, & z\not\in\mathbb{N}_-,\quad w,z-w\not\in\mathbb{N}_-;\\
0, & z\not\in\mathbb{N}_-,\quad w\in\mathbb{N}_- \text{ or } z-w\in\mathbb{N}_-;\\
\frac{\langle z\rangle_w}{w!},& z\in\mathbb{N}_-, \quad w\in\mathbb{N}_0;\\
\frac{\langle z\rangle_{z-w}}{(z-w)!}, & z,w\in\mathbb{N}_-, \quad z-w\in\mathbb{N}_0;\\
0, & z,w\in\mathbb{N}_-, \quad z-w\in\mathbb{N}_-;\\
\infty, & z\in\mathbb{N}_-, \quad w\not\in\mathbb{Z}.
\end{dcases}
\end{equation*}
\par
The Stirling numbers of the first kind $s(n,k)$ for $n\ge k\ge0$ can be generated~\cite[p.~20, (1.30)]{Temme-96-book} by
\begin{equation*}
\frac{[\ln(1+x)]^k}{k!}=\sum_{n=k}^\infty s(n,k)\frac{x^n}{n!},\quad |x|<1
\end{equation*}
and satisfy diagonal recursive relations
\begin{equation*}
\frac{s(n+k,k)}{\binom{n+k}{k}}
=\sum_{\ell=0}^{n} (-1)^{\ell}\frac{\langle k\rangle_{\ell}}{\ell!} \sum_{m=0}^\ell(-1)^m\binom{\ell}{m}\frac{s(n+m,m)}{\binom{n+m}{m}}
\end{equation*}
and
\begin{align*}
s(n,k)&=(-1)^{k}\sum_{m=1}^{n}(-1)^{m}\sum_{\ell=k-m}^{k-1}(-1)^{\ell} \binom{n}{\ell}\binom{\ell}{k-m} s(n-\ell,k-\ell)\\ 
&=(-1)^{n-k}\sum_{\ell=0}^{k-1}(-1)^\ell\binom{n}{\ell} \binom{\ell-1}{k-n-1}s(n-\ell,k-\ell)
\end{align*}
in~\cite[p.~23, Theorem~1.1]{1st-Stirl-No-adjust.tex} and~\cite[p.~156, Theorem~4]{AADM-2821.tex}.
\par
The partial Bell polynomials, or say, the Bell polynomials of the second kind, can be denoted and defined by
\begin{equation*}
\bell_{n,k}(x_1,x_2,\dotsc,x_{n-k+1})=\sum_{\substack{1\le i\le n-k+1\\ \ell_i\in\{0\}\cup\mathbb{N}\\ \sum_{i=1}^{n-k+1}i\ell_i=n\\
\sum_{i=1}^{n-k+1}\ell_i=k}}\frac{n!}{\prod_{i=1}^{n-k+1}\ell_i!} \prod_{i=1}^{n-k+1}\biggl(\frac{x_i}{i!}\biggr)^{\ell_i}.
\end{equation*}
in~~\cite[p.~412, Definition~11.2]{Charalambides-book-2002} and~\cite[p.~134, Theorem~A]{Comtet-Combinatorics-74}.

\section{Motivations}

Let $f(z)$ and $h(z)$ be infinitely differentiable functions such that the function $f(z)$ has the formal series expansion $f(z)=\sum_{k=0}^{\infty}c_kx^k$ and the composite function $h(f(z))$ is defined on a non-empty open interval. A natural problem is to find the series expansion of the composite function $h(f(z))$. This problem can be regarded as how to compute derivatives of the composite function $h(f(z))$. There have been a long history and a number of literature in textbooks, handbooks, monographs, and research articles on this problem. See the references~\cite{abram, Hansen-B-1975, Bell-value-elem-funct.tex, AJOM-D-16-00138.tex, Temme-96-book}, for example.
\par
In the above general theory, the cases $h(z)=z^r$ for $r\in\mathbb{C}\setminus\{1\}$ and $f(z)$ being concrete elementary functions are of special interest and attract some mathematicians. We recall some results as follows.
\par
In~\cite[p.~377, (3.5)]{Thir-Nanj-1951-India} and~\cite[pp.~109--110, Lemma~1]{Yang-Zhen-MIA-2018-Bessel}, it was obtained that
\begin{equation*}
I_\mu(x)I_\nu(x)=\frac{1}{\Gamma(\mu+1)\Gamma(\nu+1)} \sum_{n=0}^{\infty} \frac{(\mu+\nu+n+1)_n}{n!(\mu+1)_n(\nu+1)_n} \biggl(\frac{x}{2}\biggr)^{2n+\mu+\nu},
\end{equation*}
where the first kind modified Bessel function $I_\nu(z)$ can be represented~\cite[p.~375, 9.6.10]{abram} by
\begin{equation*}
I_\nu(z)= \sum_{n=0}^\infty\frac1{n!\Gamma(\nu+n+1)}\biggl(\frac{z}2\biggr)^{2n+\nu}, \quad z\in\mathbb{C}.
\end{equation*}
In~\cite[p.~310]{Bender-Brody-Meister-JMP-2003}, the power series expansion
\begin{equation*}
[I_\nu(z)]^2=\sum_{k=0}^{\infty}\frac{1}{[\Gamma(\nu+k+1)]^2}\binom{2k+2\nu}{k}\biggl(\frac{z}{2}\biggr)^{2k+2\nu}
\end{equation*}
was listed.
As for the series expansion of the function $[I_\nu(z)]^r$ for $\nu\in\mathbb{C}\setminus\{-1,-2,\dotsc\}$ and $r,z\in\mathbb{C}$, please refer to~\cite{Baricz-AML-2010, Bender-Brody-Meister-JMP-2003, Bess-Pow-Polyn-CMES.tex, Howard-Fibonacci-1985, Moll-Vignat-IJNT-2014}. One of the reasons why ones investigated the series expansions of the functions $[I_\nu(z)]^r$ is that the products of the (modified) Bessel functions of the first kind appear frequently in problems of statistical mechanics and plasma physics, see~\cite{Baker-Temme-JMP-1984, Newberger-Erratum-JMP-1983, Newberger-JMP-1982}.
\par
In the papers~\cite{Borwein-Chamberland-IJMMS-2007, Brychkov-ITSF-2009, Ser-Pow-Arcs-Arctan.tex, AIMS-Math20210491.tex, PAM-Jul-06-2021-0023.tex, Tan-Der-App-Thanks.tex, Qi-Chen-Lim-RNA.tex, Bell-value-elem-funct.tex}, Maclaurin's series expansions of the powers
\begin{gather*}
\biggl(\frac{\arcsin z}{z}\biggr)^{m},\quad
\frac{(\arcsin z)^{m}}{\sqrt{1-z^2}\,},\quad
\biggl(\frac{\arcsinh z}{z}\biggr)^m,\quad
\frac{(\arcsinh z)^{m}}{\sqrt{1+z^2}\,},\\
(\arctan z)^m,\quad
(\arctanh z)^m,\quad
\sin^mz,\quad
\cos^mz, \\
\tan^mz,\quad
\cot^mz,\quad
\sec^mz,\quad
\csc^mz
\end{gather*}
for $m\ge2$ and their history were reviewed, surveyed, established, discussed, and applied. Here now we recite the following two series expansions.

\begin{thm}[{\cite[Theorem~2.1]{Ser-Pow-Arcs-Arctan.tex}}] \label{arcsin-series-expansion-unify-thm}
For $k\in\mathbb{N}$ and $|x|<1$, the function $\bigl(\frac{\arcsin x}{x}\bigr)^{k}$, whose value at $x=0$ is defined to be $1$, has Maclaurin's series expansion
\begin{equation}\label{arcsin-series-expansion-unify}
\biggl(\frac{\arcsin x}{x}\biggr)^{k}
=1+\sum_{m=1}^{\infty} (-1)^m\frac{Q(k,2m)}{\binom{k+2m}{k}}\frac{(2x)^{2m}}{(2m)!},
\end{equation}
where
\begin{equation}\label{Q(m-k)-sum-dfn}
Q(k,m)=\sum_{\ell=0}^{m} \binom{k+\ell-1}{k-1} s(k+m-1,k+\ell-1)\biggl(\frac{k+m-2}{2}\biggr)^{\ell}
\end{equation}
for $k\in\mathbb{N}$ and $m\ge2$.
\end{thm}

\begin{thm}[{\cite[Theorem~5.1]{Ser-Pow-Arcs-Arctan.tex}}] \label{arcsinh-identity-thm}
For $k\in\mathbb{N}$ and $|x|<\infty$, the function $\bigl(\frac{\arcsinh x}{x}\bigr)^{k}$, whose value at $x=0$ is defined to be $1$, has Maclaurin's series expansion
\begin{equation}\label{arcsinh-series-expansion}
\biggl(\frac{\arcsinh x}{x}\biggr)^k
=1+\sum_{m=1}^{\infty}\frac{Q(k,2m)}{\binom{k+2m}{k}} \frac{(2x)^{2m}}{(2m)!},
\end{equation}
where $Q(k,2m)$ is given by~\eqref{Q(m-k)-sum-dfn}.
\end{thm}

In the papers~\cite{Ser-Pow-Arcs-Arctan.tex, AIMS-Math20210491.tex}, the series expansion~\eqref{arcsin-series-expansion-unify} has been applied to derive closed-form formulas for specific partial Bell polynomials and to establish series representations of the generalized logsine function. These results are needed and considered in~\cite{Davydychev-Kalmykov-2001, Kalmykov-Sheplyakov-lsjk-2005, Oertel-arXiv-2010.00746v2} respectively.
\par
In the community of mathematics, the circular constant Pi has attracted a number of mathematicians spending long time and utilizing many methods to calculate it. The setting-up of the international Pi Day is the best demonstration of the importance of the circular constant Pi. Taking $x=\frac{\sqrt{2}\,}{2}$ in~\eqref{arcsin-series-expansion-unify} produces the series representation
\begin{equation}\label{arcsinh-series-pi}
\biggl(\frac{\pi}{2\sqrt{2}\,}\biggr)^{k}
=1+k!\sum_{m=1}^{\infty} (-1)^m2^m\frac{Q(k,2m)}{(k+2m)!}.
\end{equation}
\par
In this paper, by virtue of expansions of two finite products of finitely many square sums
\begin{equation*}
\prod_{\ell=1}^{k}\bigl(\ell^2+\alpha^2\bigr) \quad\text{and}\quad
\prod_{\ell=1}^{k}\bigl[(2\ell-1)^2+\alpha^2\bigr]
\end{equation*}
for $k\in\mathbb{N}$ in Lemmas~\ref{square-sum-prod-lem} and~\ref{square-sum-prod-odd-lem} below, with the aid of Taylor's series expansions around $x=1^-$ of the functions $\cosh(\alpha\arccos x)$ and $\cos(\alpha\arccos x)$ in Lemma~\ref{cosh-sinh-arcsin-arccos-thm}, and in the light of properties of partial Bell polynomials selected in Lemma~\ref{Partial-Bell-Conclusions-Lem}, we will
\begin{enumerate}
\item
establish Taylor's series expansions around $x=1^-$ of the functions
\begin{equation*}
\biggl[\frac{(\arccos x)^{2}}{2(1-x)}\biggr]^\alpha \quad\text{and}\quad \biggl[\frac{(\arccosh x)^{2}}{2(1-x)}\biggr]^k
\end{equation*}
for $\alpha\in\mathbb{R}$ and $k\in\mathbb{N}$ in terms of $Q(k,m)$ in Theorems~\ref{arccos-ser-expan-thm} and~\ref{arccos-Taylor-alpha-THM};
\item
present an explicit formula of the specific partial Bell polynomials
\begin{equation*}
\bell_{m,k}\biggl(-\frac{1}{12},\frac{2}{45},-\frac{3}{70}, \frac{32}{525},-\frac{80}{693}, \dotsc, \frac{(2m-2k+2)!!}{(2m-2k+4)!} Q(2,2m-2k+2)\biggr)
\end{equation*}
for $m\ge k\in\mathbb{N}$ in Theorem~\ref{arccos-squar-(1-x)-bell-thm};
\item
derive several combinatorial identities involving the Stirling numbers of the first kind $s(n,k)$ in Lemmas~\ref{square-sum-prod-lem} and~\ref{square-sum-prod-odd-lem}, in Corollaries~\ref{comb-id-arccos-ven=pow-cor} and~\ref{Last-v2-comb-id-arcos-Cor}, and in the proof of Theorem~\ref{arccos-squar-(1-x)-bell-thm};
\item
demonstrate several series representations of the circular constant $\pi$ and its powers $\pi^\alpha$ for $\alpha\in\mathbb{R}$ in Corollaries~\ref{series-approx-pi-first2cor} and~\ref{Pi-Taylor-alpha-Cor};
\item
recover Maclaurin's series expansions~\eqref{arcsin-series-expansion-unify} and~\eqref{arcsinh-series-expansion} in Theorem~\ref{arcsin-series-expansion-unify-thm} and~\ref{arcsinh-identity-thm}; and
\item
also deduce other useful, meaningful, and significant conclusions in Corollaries~\ref{Minus-sh-square-ser-COR}, \ref{COR-4-Deriv}, \ref{arccosh-squ-deriv-form2-cor}, \ref{arccosH-square-ser-expan-thm}, and~\ref{arccos-odd-power-deriv-thm}, in Remarks~\ref{Rem3.3Q(m-k)-Probl} and~\ref{Rem4.1-Pi-Speed}, and elsewhere in this paper.
\end{enumerate}

\section{Important lemmas}

For attaining our aims mentioned just now, we need the following four important lemmas.

\begin{lem}\label{square-sum-prod-lem}
For $k\in\mathbb{N}$ and $\alpha\in\mathbb{C}\setminus\{0\}$, we have
\begin{equation}\label{square-sum-first-stirl-eq}
\prod_{\ell=1}^{k}\bigl(\ell^2+\alpha^2\bigr)
=(-1)^{k}\sum_{j=0}^{k}(-1)^{j}\Biggl[\sum_{\ell=2j+1}^{2k+1}\binom{\ell}{2j+1}s(2k+1,\ell)k^{\ell-2j-1}\Biggr]\alpha^{2j}
\end{equation}
and
\begin{equation}\label{square-sum-first-alpha=0}
\sum_{\ell=0}^{2k}(\ell+1) s(2k+1,\ell+1) k^{\ell}=(-1)^{k}(k!)^2.
\end{equation}
For $0\le j\le k-1$, we have
\begin{equation}\label{Stirl-first-ID}
\sum_{\ell=2j+1}^{2k-1}\binom{\ell}{2j}s(2k-1,\ell)(k-1)^{\ell-2j}=-s(2k-1,2j).
\end{equation}
\end{lem}

\begin{proof}
In~\cite[p.~165, (12.1)]{Quaintance-Gould-Stirling-B}, there exists the formula
\begin{equation}\label{(12.1)-Quaintance}
n!\binom{z}{n}=\sum_{k=0}^{n}s(n,k)z^k, \quad z\in\mathbb{C}, \quad n\ge0.
\end{equation}
\par
It is not difficult to verify that
\begin{equation}\label{square-sum-prod-id}
\begin{aligned}
\prod_{\ell=1}^{k}\bigl[(\ell-1)^2+\alpha^2\bigr]
&=\prod_{\ell=0}^{k-1}\bigl(\ell^2+\alpha^2\bigr)\\
&=(-1)^k\frac{(\ti \alpha)_k}{(\ti \alpha+1)_{-k}}\\
&=(-1)^k\ti \alpha\frac{\Gamma(\ti \alpha+k)}{\Gamma(\ti \alpha-k+1)}\\
&=(-1)^k\ti \alpha(2k-1)!\binom{\ti \alpha+k-1}{2k-1}
\end{aligned}
\end{equation}
for $\alpha\in\mathbb{C}\setminus\{0\}$ and $k\in\mathbb{N}$, where $\ti=\sqrt{-1}\,$ is the imaginary unit and $\alpha\in\mathbb{C}$.
Substituting the formula~\eqref{(12.1)-Quaintance} into~\eqref{square-sum-prod-id} results in
\begin{align*}
\prod_{\ell=0}^{k-1}\bigl(\ell^2+\alpha^2\bigr)
&=(-1)^k\ti \alpha\sum_{\ell=0}^{2k-1}s(2k-1,\ell)(\ti \alpha+k-1)^\ell\\
&=(-1)^k\ti \alpha\sum_{\ell=0}^{2k-1}s(2k-1,\ell)\sum_{j=0}^{\ell}\binom{\ell}{j}(\ti \alpha)^j(k-1)^{\ell-j}\\
&=(-1)^k\sum_{j=0}^{2k-1}\Biggl[\sum_{\ell=j}^{2k-1}\binom{\ell}{j}s(2k-1,\ell)(k-1)^{\ell-j}\Biggr](\ti \alpha)^{j+1}\\
&=(-1)^k\sum_{j=0}^{2k-1}\Biggl[\sum_{\ell=j}^{2k-1}\binom{\ell}{j}s(2k-1,\ell)(k-1)^{\ell-j}\Biggr]\alpha^{j+1} \cos\frac{(j+1)\pi}2\\
&\quad+\ti(-1)^k\sum_{j=0}^{2k-1}\Biggl[\sum_{\ell=j}^{2k-1}\binom{\ell}{j}s(2k-1,\ell)(k-1)^{\ell-j}\Biggr]\alpha^{j+1} \sin\frac{(j+1)\pi}2
\end{align*}
for $\alpha\in\mathbb{C}\setminus\{0\}$ and $k\in\mathbb{N}$, where we used the identity
\begin{equation}\label{imag-power-k-id}
\ti^k=\cos\frac{k\pi}2+\ti\sin\frac{k\pi}2, \quad k\ge0.
\end{equation}
As a result, equating the real and imaginary parts, we obtain
\begin{equation}\label{square-sum-stirl-eq}
\prod_{\ell=0}^{k-1}\bigl(\ell^2+\alpha^2\bigr)
=(-1)^k\sum_{j=0}^{2k-1}\cos\frac{(j+1)\pi}2\Biggl[s(2k-1,j) +\sum_{\ell=j+1}^{2k-1}\binom{\ell}{j}s(2k-1,\ell)(k-1)^{\ell-j}\Biggr]\alpha^{j+1}
\end{equation}
and
\begin{equation}\label{imag-part-eq}
\sum_{j=0}^{2k-1}\sin\frac{(j+1)\pi}2\Biggl[s(2k-1,j) +\sum_{\ell=j+1}^{2k-1}\binom{\ell}{j}s(2k-1,\ell)(k-1)^{\ell-j}\Biggr]\alpha^{j+1} =0.
\end{equation}
\par
Since $\cos(j\pi)=(-1)^j$ and $\cos\frac{(2j+1)\pi}2=0$ for $j\in\mathbb{Z}$, the equality~\eqref{square-sum-stirl-eq} can be simplified as~\eqref{square-sum-first-stirl-eq}.
\par
Taking $\alpha\to0$ in~\eqref{square-sum-first-stirl-eq} reduces to~\eqref{square-sum-first-alpha=0}.
\par
Since $\sin(j\pi)=0$ and $\sin\frac{(2j+1)\pi}2=(-1)^j$ for $j\in\mathbb{Z}$, the equality~\eqref{imag-part-eq} becomes
\begin{equation*}
\sum_{j=0}^{k-1}(-1)^j\Biggl[s(2k-1,2j) +\sum_{\ell=2j+1}^{2k-1}\binom{\ell}{2j}s(2k-1,\ell)(k-1)^{\ell-2j}\Biggr]\alpha^{2j+1} =0.
\end{equation*}
Further regarding $\alpha$ as a variable leads to
\begin{equation*}
(-1)^j\Biggl[s(2k-1,2j) +\sum_{\ell=2j+1}^{2k-1}\binom{\ell}{2j}s(2k-1,\ell)(k-1)^{\ell-2j}\Biggr]=0
\end{equation*}
which is equivalent to the combinatorial identity~\eqref{Stirl-first-ID}.
The proof of Lemma~\ref{square-sum-prod-lem} is complete.
\end{proof}

\begin{lem}\label{square-sum-prod-odd-lem}
For $k\in\mathbb{N}$ and $\alpha\in\mathbb{C}\setminus\{0\}$, we have
\begin{equation}\label{square-sum-odd-id}
\prod_{\ell=1}^{k}\bigl[(2\ell-1)^2+\alpha^2\bigr]
=(-1)^{k}2^{2k} \sum_{j=0}^{2k} (-1)^j \Biggl[\sum_{\ell=2j}^{2k}\frac{s(2k,\ell)}{2^\ell}\binom{\ell}{2j} (2k-1)^{\ell-2j}\Biggr]\alpha^{2j} \end{equation}
and
\begin{equation}\label{square-sum-odd-alpha=0}
\sum_{\ell=0}^{2k}s(2k,\ell)\biggl(k-\frac{1}{2}\biggr)^{\ell}
=(-1)^{k}\biggl[\frac{(2k-1)!!}{2^{k}}\biggr]^2.
\end{equation}
For $0\le j<k\in\mathbb{N}$, we have
\begin{equation}\label{first-stirl-2id}
\sum_{\ell=2j+1}^{2k}\binom{\ell}{2j+1} s(2k,\ell)\biggl(k-\frac12\biggr)^{\ell}=0.
\end{equation}
\end{lem}

\begin{proof}
The identities in~\eqref{square-sum-prod-id} can be rearranged as
\begin{align*}
\prod_{\ell=1}^{k}\bigl(\ell^2+\alpha^2\bigr)
&=(-1)^{k+1}\frac{(\ti \alpha)_{k+1}}{\alpha^2(\ti \alpha+1)_{-(k+1)}}\\
&=(-1)^{k+1}\ti \frac{\Gamma(\ti \alpha+k+1)}{\alpha\Gamma(\ti \alpha-k)}\\
&=(-1)^{k+1}\ti\frac{(2k+1)!}{\alpha}\binom{\ti \alpha+k}{2k+1}
\end{align*}
for $\alpha\in\mathbb{C}\setminus\{0\}$ and $k\in\mathbb{N}$.
By this, we acquire
\begin{equation}
\begin{aligned}\label{suqare-sum-gamma}
\prod_{\ell=1}^{k}\bigl[(2\ell-1)^2+\alpha^2\bigr]
&=\frac{\prod_{\ell=1}^{2k}\bigl(\ell^2+\alpha^2\bigr)} {\prod_{\ell=1}^{k}[\alpha^2+(2\ell)^2]}\\
&=\frac{\prod_{\ell=1}^{2k}\bigl(\ell^2+\alpha^2\bigr)} {4^k\prod_{\ell=1}^{k}[(\alpha/2)^2+\ell^2]}\\
&=\frac{1}{4^k} (-1)^{2k+1}\ti \frac{\Gamma(\ti \alpha+2k+1)}{\alpha\Gamma(\ti \alpha-2k)} \frac{1}{(-1)^{k+1}\ti} \frac{\alpha\Gamma(\ti \alpha/2-k)}{2\Gamma(\ti \alpha/2+k+1)}\\
&=\frac{(-1)^{k}}{2^{2k+1}} \frac{\Gamma(\ti \alpha/2-k)}{\Gamma(\ti \alpha-2k)} \frac{\Gamma(\ti \alpha+2k+1)}{\Gamma(\ti \alpha/2+k+1)}
\end{aligned}
\end{equation}
for $\alpha\in\mathbb{C}\setminus\{0\}$ and $k\in\mathbb{N}$.
Making use of the Gauss multiplication formula
\begin{equation*}
\Gamma(nz)=\frac{n^{nz-1/2}}{(2\pi)^{(n-1)/2}}\prod_{k=0}^{n-1}\Gamma\biggl(z+\frac{k}n\biggr)
\end{equation*}
in~\cite[p.~256, 6.1.20]{abram} leads to
\begin{align*}
\frac{\Gamma(\ti \alpha-2k)}{\Gamma(\ti \alpha/2-k)}
&=\frac{1}{\Gamma(\ti \alpha/2-k)} \frac{2^{\ti \alpha-2k-1/2}}{(2\pi)^{1/2}}\Gamma\biggl(\frac{\ti \alpha}{2}-k\biggr) \Gamma\biggl(\frac{\ti \alpha}{2}-k+\frac{1}2\biggr)\\
&=\frac{2^{\ti \alpha-2k-1}}{\sqrt{\pi}\,} \Gamma\biggl(\frac{1+\ti \alpha}{2}-k\biggr)
\end{align*}
and
\begin{align*}
\frac{\Gamma(\ti \alpha+2k+1)}{\Gamma(\ti \alpha/2+k+1)}
&=\frac{(\ti \alpha+2k)}{\Gamma(\ti \alpha/2+k+1)}\Gamma(2(\ti \alpha/2+k))\\
&=\frac{(\ti \alpha+2k)}{\Gamma(\ti \alpha/2+k+1)} \frac{2^{2(\ti \alpha/2+k)-1/2}}{(2\pi)^{1/2}}\Gamma(\ti \alpha/2+k)\Gamma\biggl(\ti \alpha/2+k+\frac{1}2\biggr)\\
&=\frac{2^{\ti \alpha+2k}}{\sqrt{\pi}\,}\Gamma\biggl(\frac{1+\ti \alpha}{2}+k\biggr)
\end{align*}
for $\alpha\in\mathbb{C}\setminus\{0\}$ and $k\in\mathbb{N}$.
Substituting these two equalities into~\eqref{suqare-sum-gamma}, utilizing the formula~\eqref{(12.1)-Quaintance}, interchanging the order of double sums, and employing the identity~\eqref{imag-power-k-id} yield
\begin{align*}
\prod_{\ell=1}^{k}\bigl[(2\ell-1)^2+\alpha^2\bigr]
&=\frac{(-1)^{k}}{2^{2k+1}}\frac{\frac{2^{\ti \alpha+2k}}{\sqrt{\pi}\,}\Gamma\bigl(\frac{1+\ti \alpha}{2}+k\bigr)} {\frac{2^{\ti \alpha-2k-1}}{\sqrt{\pi}\,} \Gamma\bigl(\frac{1+\ti \alpha}{2}-k\bigr)}\\
&=(-1)^{k}2^{2k} \frac{\Gamma\bigl(\frac{1+\ti \alpha}{2}+k\bigr)} {\Gamma\bigl(\frac{1+\ti \alpha}{2}-k\bigr)}\\
&=(-1)^{k}2^{2k}(2k)!\binom{\frac{\ti \alpha}{2}+k-\frac{1}{2}}{2k}\\
&=(-1)^{k}2^{2k}\sum_{\ell=0}^{2k}s(2k,\ell)\biggl(\frac{\ti \alpha}{2}+k-\frac{1}{2}\biggr)^\ell\\
&=(-1)^{k}2^{2k}\sum_{\ell=0}^{2k}\frac{s(2k,\ell)}{2^\ell}\sum_{j=0}^{\ell}\binom{\ell}{j}(\ti \alpha)^j (2k-1)^{\ell-j}\\
&=(-1)^{k}2^{2k}\sum_{j=0}^{2k} \Biggl[\sum_{\ell=j}^{2k}\frac{s(2k,\ell)}{2^\ell}\binom{\ell}{j} (2k-1)^{\ell-j}\Biggr](\ti \alpha)^j\\
&=(-1)^{k}2^{2k}\sum_{j=0}^{2k} \Biggl[\sum_{\ell=j}^{2k}\frac{s(2k,\ell)}{2^\ell}\binom{\ell}{j} (2k-1)^{\ell-j}\Biggr]\biggl(\cos\frac{j\pi}2\biggr)\alpha^j\\
&\quad+\ti(-1)^{k}2^{2k}\sum_{j=0}^{2k} \Biggl[\sum_{\ell=j}^{2k}\frac{s(2k,\ell)}{2^\ell}\binom{\ell}{j} (2k-1)^{\ell-j}\Biggr]\biggl(\sin\frac{j\pi}2\biggr)\alpha^j
\end{align*}
for $\alpha\in\mathbb{C}\setminus\{0\}$ and $k\in\mathbb{N}$.
From the facts that $\cos(j\pi)=(-1)^j$ and $\cos\frac{(2j+1)\pi}2=0$ for $j\in\mathbb{Z}$, equating the above real part and simplifying produce the identity~\eqref{square-sum-odd-id}.
\par
Taking $\alpha\to0$ in~\eqref{square-sum-odd-id} reduces to~\eqref{square-sum-odd-alpha=0}.
\par
From the facts that $\sin(j\pi)=0$ and $\sin\frac{(2j+1)\pi}2=(-1)^j$ for $j\in\mathbb{Z}$, the last imaginary part becomes
\begin{equation*}
\sum_{j=0}^{k-1} (-1)^j\Biggl[\sum_{\ell=2j+1}^{2k}\frac{s(2k,\ell)}{2^\ell}\binom{\ell}{2j+1} (2k-1)^{\ell-2j-1}\Biggr]\alpha^{2j+1}=0.
\end{equation*}
Regarding $\alpha$ as a variable means the identity~\eqref{first-stirl-2id}.
The proof of Lemma~\ref{square-sum-prod-odd-lem} is complete.
\end{proof}

\begin{rem}
The formula~\eqref{(12.1)-Quaintance} can be reformulated as
\begin{equation}\label{rising-s(n-k)-gen}
(z)_n=\sum_{k=0}^{n}(-1)^{n-k}s(n,k)z^k.
\end{equation}
Comparing this with the definition~\eqref{rising-Factorial}, we can regard~\eqref{square-sum-first-stirl-eq} and~\eqref{square-sum-odd-id} in Lemmas~\ref{square-sum-prod-lem} and~\ref{square-sum-prod-odd-lem} as generalizations of the formula~\eqref{rising-s(n-k)-gen}.
\end{rem}

\begin{rem}
The identity~\eqref{square-sum-odd-alpha=0} is a special case of the identity~\eqref{(12.1)-Quaintance} or~\eqref{rising-s(n-k)-gen}.
\par
In the monograph~\cite{Quaintance-Gould-Stirling-B}, we do not find the combinatorial identities~\eqref{square-sum-first-alpha=0}, \eqref{Stirl-first-ID}, and~\eqref{first-stirl-2id}.
\end{rem}

\begin{rem}\label{Rem3.3Q(m-k)-Probl}
The combinatorial identities~\eqref{square-sum-first-alpha=0} and~\eqref{square-sum-odd-alpha=0} can be rearranged as
\begin{equation}\label{square-sum-first-alpha=0-Q}
Q(2,2k)=(-1)^{k}(k!)^2, \quad k\in\mathbb{N}
\end{equation}
and
\begin{equation}\label{square-sum-odd-alpha=0-Q}
Q(1,2k)
=(-1)^{k}\biggl[\frac{(2k-1)!!}{2^{k}}\biggr]^2, \quad k\in\mathbb{N}.
\end{equation}
\par
Due to the trivial result $s(n,n)=1$ for $n\in\mathbb{N}_0$, the combinatorial identities~\eqref{Stirl-first-ID} and~\eqref{first-stirl-2id} can be rearranged as
\begin{equation*}
s(2j+1,2j)=-j(2j+1),\quad
Q(2j+1,2m-1)=0
\end{equation*}
for $m\ge2$ and $j\in\mathbb{N}_0$, and
\begin{equation*}
\sum_{\ell=1}^{2m}\binom{2j+\ell}{2j+1} s(2j+2m,2j+\ell)\biggl(j+m-\frac12\biggr)^{\ell}=0
\end{equation*}
for $j\in\mathbb{N}_0$ and $m\in\mathbb{N}$.
On the other hand, we have
\begin{gather*}
Q(2j,2m)=s(2j+2m-1,2j-1) +2j(j+m-1) s(2j+2m-1,2j)\\
+2j(j+m-1)\sum_{\ell=1}^{2m-1}\frac{1}{\ell+1} \binom{2j+\ell}{2j} s(2j+2m-1,2j+\ell)(j+m-1)^{\ell}
\end{gather*}
for $j,m\in\mathbb{N}$. Can one give a simple form for the quantity
\begin{equation*}
\sum_{\ell=1}^{2m-1}\frac{1}{\ell+1} \binom{2j+\ell}{2j} s(2j+2m-1,2j+\ell)(j+m-1)^{\ell}
\end{equation*}
for $j,m\in\mathbb{N}$?
Can one discover more simple forms, similar to $Q(1,2k)$ and $Q(2,2k)$ in~\eqref{square-sum-first-alpha=0-Q} and~\eqref{square-sum-odd-alpha=0-Q}, of $Q(k,m)$ for some $k\in\mathbb{N}$ and $m\ge2$?
\end{rem}

\begin{lem}\label{cosh-sinh-arcsin-arccos-thm}
For $\alpha\in\mathbb{C}\setminus\{0\}$ and $|x|<1$, we have
\begin{align}
\cosh(\alpha\arcsin x)&=\sum_{k=0}^{\infty}\Biggl(\prod_{\ell=1}^{k}\bigl[4(\ell-1)^2+\alpha^2\bigr]\Biggr)\frac{x^{2k}}{(2k)!}, \label{cosh-arcsin-ser}\\
\sinh(\alpha\arcsin x)&=\alpha\sum_{k=0}^{\infty} \Biggl(\prod_{\ell=1}^{k}\bigl[(2\ell-1)^2+\alpha^2\bigr]\Biggr) \frac{x^{2k+1}}{(2k+1)!}, \label{sinh-arcsin-ser}\\
\cosh(\alpha\arccos x)&=\sum_{k=0}^{\infty}\frac{(-1)^k}{(2k-1)!!} \Biggl(\prod_{\ell=1}^{k}\bigl[(\ell-1)^2+\alpha^2\bigr]\Biggr) \frac{(x-1)^k}{k!} \label{cosh-arccos-ser}\\
\begin{split}\label{cosh-arccos-ser=0}
&=\biggl(\cosh\frac{\alpha\pi}{2}\biggr) \sum_{k=0}^{\infty} \Biggl(\prod_{\ell=1}^{k}\bigl[4(\ell-1)^2+\alpha^2\bigr]\Biggr) \frac{x^{2k}}{(2k)!}\\
&\quad-\alpha\biggl(\sinh\frac{\alpha\pi}{2}\biggr) \sum_{k=0}^{\infty} \Biggl(\prod_{\ell=1}^{k}\bigl[(2\ell-1)^2+\alpha^2\bigr]\Biggr) \frac{x^{2k+1}}{(2k+1)!},
\end{split}\\
\begin{split}
\sinh(\alpha\arccos x)&=\biggl(\sinh\frac{\alpha\pi}{2}\biggr) \sum_{k=0}^{\infty}\Biggl(\prod_{\ell=1}^{k} \bigl[4(\ell-1)^2+\alpha^2\bigr]\Biggr) \frac{x^{2k}}{(2k)!}\\
&\quad -\alpha\biggl(\cosh\frac{\alpha\pi}{2}\biggr) \sum_{k=0}^{\infty} \Biggl(\prod_{\ell=1}^{k}\bigl[(2\ell-1)^2+\alpha^2\bigr]\Biggr) \frac{x^{2k+1}}{(2k+1)!}, \label{sinh-arccos-ser}
\end{split}\\
\cos(\alpha\arcsin x)&=\sum_{k=0}^{\infty}\Biggl(\prod_{\ell=1}^{k}\bigl[4(\ell-1)^2-\alpha^2\bigr]\Biggr)\frac{x^{2k}}{(2k)!}, \label{cos-arcsin-ser}\\
\sin(\alpha\arcsin x)&=\alpha\sum_{k=0}^{\infty} \Biggl(\prod_{\ell=1}^{k}\bigl[(2\ell-1)^2-\alpha^2\bigr]\Biggr) \frac{x^{2k+1}}{(2k+1)!}, \label{sin-arcsin-ser}\\
\cos(\alpha\arccos x)&=\sum_{k=0}^{\infty} \frac{(-1)^k}{(2k-1)!!} \Biggl(\prod_{\ell=1}^{k-1}\bigl[\bigl(\ell-1)^2-\alpha^2\bigr]\Biggr) \frac{(x-1)^k}{k!} \label{cos-arccos-ser}\\
\begin{split}\label{cos-arccos-ser-x=0}
&=\biggl(\cos\frac{\alpha\pi}{2}\biggr) \sum_{k=0}^{\infty} \Biggl(\prod_{\ell=1}^{k}\bigl[4(\ell-1)^2-\alpha^2\bigr]\Biggr) \frac{x^{2k}}{(2k)!}\\
&\quad+\alpha\biggl(\sin\frac{\alpha\pi}{2}\biggr)\sum_{k=0}^{\infty} \Biggl(\prod_{\ell=1}^{k}\bigl[(2\ell-1)^2-\alpha^2\bigr]\Biggr)\frac{x^{2k+1}}{(2k+1)!},
\end{split}\\
\begin{split}\label{sin-arccos-ser}
\sin(\alpha\arccos x)&=\biggl(\sin\frac{\alpha\pi}{2}\biggr) \sum_{k=0}^{\infty}\Biggl(\prod_{\ell=1}^{k}\bigl[4(\ell-1)^2-\alpha^2\bigr]\Biggr)\frac{x^{2k}}{(2k)!}\\
&\quad-\alpha\biggl(\cos\frac{\alpha\pi}{2}\biggr) \sum_{k=0}^{\infty} \Biggl(\prod_{\ell=1}^{k}\bigl[(2\ell-1)^2-\alpha^2\bigr]\Biggr) \frac{x^{2k+1}}{(2k+1)!},
\end{split}
\end{align}
where $(-1)!!=1$ and any empty product is understood to be $1$.
\end{lem}

\begin{proof}
Let $f_\alpha(x)=\cosh(\alpha\arcsin x)$. Then consecutive differentiations and simplifications give
\begin{gather*}
f_\alpha'(x)=\frac{\alpha}{\sqrt{1-x^2}\,} \sinh(\alpha\arcsin x)
=\frac{\alpha}{\sqrt{1-x^2}\,} \sqrt{f_\alpha^2(x)-1}\,,\\
\bigl(1-x^2\bigr)\bigl[f_\alpha'(x)\bigr]^2-\alpha^2f_\alpha^2(x)+\alpha^2=0,\\
\bigl(1-x^2\bigr)f_\alpha''(x)-xf_\alpha'(x)-\alpha^2f_\alpha(x)=0,\\
\bigl(1-x^2\bigr)f_\alpha^{(3)}(x)-3xf_\alpha''(x)-\bigl(1+\alpha^2\bigr)f_\alpha'(x)=0.
\end{gather*}
Accordingly, the differential equation
\begin{equation}\label{n-deriv-eq-gen}
\bigl(1-x^2\bigr)f_\alpha^{(k+2)}(x)-(2k+1)xf_\alpha^{(k+1)}(x)-\bigl(k^2+\alpha^2\bigr)f_\alpha^{(k)}(x)=0
\end{equation}
is valid for $k=0,1$ respectively. Differentiating on both sides of~\eqref{n-deriv-eq-gen} results in
\begin{equation*}
\bigl(1-x^2\bigr)f_\alpha^{(k+3)}(x)-(2k+3)xf_\alpha^{(k+2)}(x) -\bigl[(k+1)^2+\alpha^2\bigr]f_\alpha^{(k+1)}(x)=0.
\end{equation*}
By induction, the equation~\eqref{n-deriv-eq-gen} is valid for all $k\ge0$.
Taking $x\to0$ in~\eqref{n-deriv-eq-gen} gives
\begin{equation}\label{n-deriv-x=0}
f_\alpha^{(k+2)}(0)-\bigl(k^2+\alpha^2\bigr)f_\alpha^{(k)}(0)=0, \quad k\ge0.
\end{equation}
It is clear that $f_\alpha(0)=1$ and $f_\alpha'(0)=0$. Substituting these two initial values into the recursive relation~\eqref{n-deriv-x=0} and consecutively recursing reveal $f_\alpha^{(2k-1)}(0)=0$ and
\begin{equation*}
f_\alpha^{(2k)}(0)=\prod_{\ell=1}^{k}\bigl[4(\ell-1)^2+\alpha^2\bigr], \quad k\in\mathbb{N}.
\end{equation*}
Consequently, the series expansion~\eqref{cosh-arcsin-ser} follows.
\par
Let $f_\alpha(x)=\sinh(\alpha\arcsin x)$. Then consecutive differentiations and simplifications give
\begin{gather*}
f_\alpha'(x)=\frac{\alpha}{\sqrt{1-x^2}\,} \cosh(\alpha\arcsin x)
=\frac{\alpha}{\sqrt{1-x^2}\,} \sqrt{f_\alpha^2(x)+1}\,,\\
\bigl(1-x^2\bigr)\bigl[f_\alpha'(x)\bigr]^2-\alpha^2f_\alpha^2(x)-\alpha^2=0,\\
\bigl(1-x^2\bigr)f_\alpha''(x)-xf_\alpha'(x)-\alpha^2f_\alpha(x)=0,\\
\bigl(1-x^2\bigr)f_\alpha^{(3)}(x)-3xf_\alpha''(x)-\bigl(1+\alpha^2\bigr)f_\alpha'(x)=0.
\end{gather*}
By the same argument as above, the derivative $f_\alpha^{(k)}(0)$ for $n\ge0$ satisfy the recursive relation~\eqref{n-deriv-x=0}. Furthermore, from the facts that $f_\alpha(0)=0$ and $f_\alpha'(0)=\alpha$, we conclude $f_\alpha^{(2k)}(0)=0$ and
\begin{equation*}
f_\alpha^{(2k+1)}(0)=\alpha\prod_{\ell=1}^{k}\bigl[(2\ell-1)^2+\alpha^2\bigr], \quad k\ge0.
\end{equation*}
Consequently, the series expansion~\eqref{sinh-arcsin-ser} follows.
\par
Let $f_\alpha(x)=\cosh(\alpha\arccos x)$. Then successively differentiating yields
\begin{gather*}
f_\alpha'(x)=-\frac{\alpha}{\sqrt{1-x^2}}\sinh(\alpha\arccos x)
=-\frac{\alpha}{\sqrt{1-x^2}}\sqrt{f_\alpha^2(x)-1}\,,\\
\bigl(1-x^2\bigr)[f_\alpha'(x)]^2-\alpha^2\bigl[f_\alpha^2(x)-1\bigr]=0,\\
\bigl(1-x^2\bigr)f_\alpha''(x)-xf_\alpha'(x)-\alpha^2f_\alpha(x)=0,\\
\bigl(1-x^2\bigr)f_\alpha'''(x)-3xf_\alpha''(x)-\bigl(1+\alpha^2\bigr)f_\alpha'(x)=0.
\end{gather*}
As argued above, we conclude that the derivatives $f_\alpha^{(k)}(x)$ for $k\ge0$ satisfy the equation~\eqref{n-deriv-eq-gen}. Letting $x\to1^-$ in~\eqref{n-deriv-eq-gen} gives
\begin{equation}\label{n-deriv-x=1}
(2k+1)f_\alpha^{(k+1)}(1)+\bigl(k^2+\alpha^2\bigr)f_\alpha^{(k)}(1)=0.
\end{equation}
Setting $x\to0$ in~\eqref{n-deriv-eq-gen} leads to~\eqref{n-deriv-x=0}
It is easy to see that
\begin{equation*}
f_\alpha(1)=1,\quad f_\alpha'(1)=-\alpha^2, \quad f_\alpha(0)=\cosh\frac{\alpha\pi}{2},\quad f_\alpha'(0)=-\alpha\sinh\frac{\alpha\pi}{2}.
\end{equation*}
Substituting these four initial values into~\eqref{n-deriv-x=1} and inductively recursing reveal
\begin{align*}
f_\alpha^{(k)}(1)&=(-1)^k\frac{\prod_{\ell=1}^{k}\bigl[(\ell-1)^2+\alpha^2\bigr]}{(2k-1)!!},\\
f_\alpha^{(2k)}(0)&=\biggl(\cosh\frac{\alpha\pi}{2}\biggr) \prod_{\ell=1}^{k}\bigl[4(\ell-1)^2+\alpha^2\bigr],
\end{align*}
and
\begin{equation*}
f_\alpha^{(2k+1)}(0)=-\alpha\biggl(\sinh\frac{\alpha\pi}{2}\biggr) \prod_{\ell=1}^{k}\bigl[(2\ell-1)^2+\alpha^2\bigr]
\end{equation*}
for $k\ge0$.
Consequently, the series expansions~\eqref{cosh-arccos-ser} and~\eqref{cosh-arccos-ser=0} follow.
\par
Let $f_\alpha(x)=\sinh(\alpha\arccos x)$. Then
\begin{gather*}
f_\alpha'(x)=-\frac{\alpha}{\sqrt{1-x^2}\,}\cosh(\alpha\arccos x)
=-\frac{\alpha}{\sqrt{1-x^2}\,}\sqrt{f_\alpha^2(x)+1}\,,\\
\bigl(1-x^2\bigr)\bigl[f_\alpha'(x)\bigr]^2-\alpha^2\bigl[f_\alpha^2(x)+1\bigr]=0,\\
\bigl(1-x^2\bigr)f_\alpha''(x)-xf_\alpha'(x)-\alpha^2f_\alpha(x)=0,
\end{gather*}
and, inductively, the derivatives $f_\alpha^{(k)}(x)$ for $k\ge0$ satisfy the equations~\eqref{n-deriv-eq-gen} and~\eqref{n-deriv-x=0}. Since $f_\alpha(0)=\sinh\frac{\alpha\pi}{2}$ and $f_\alpha'(0)=-\alpha\cosh\frac{\alpha\pi}{2}$, we obtain
\begin{equation*}
f_\alpha^{(2k)}(0)=\Biggl(\prod_{\ell=1}^{k}\bigl[4(\ell-1)^2+\alpha^2\bigr]\Biggr)\sinh\frac{\alpha\pi}{2}
\end{equation*}
and
\begin{equation*}
f_\alpha^{(2k+1)}(0)=-\alpha\Biggl(\prod_{\ell=1}^{k}\bigl[(2\ell-1)^2+\alpha^2\bigr]\Biggr)\cosh\frac{\alpha\pi}{2}
\end{equation*}
for $k\ge0$. Consequently, the series expansion~\eqref{sinh-arccos-ser} follows.
\par
Let $f_\alpha(x)=\cos(\alpha\arcsin x)$. Then
\begin{gather*}
f_\alpha'(x)=-\frac{\alpha}{\sqrt{1-x^2}\,}\sin(\alpha\arcsin x)=-\frac{\alpha}{\sqrt{1-x^2}\,}\sqrt{1-f_\alpha^2(x)}\,\\
\bigl(1-x^2\bigr)\bigl[f_\alpha'(x)\bigr]^2-\alpha^2\bigl[1-f_\alpha^2(x)\bigr]=0,\\
\bigl(1-x^2\bigr)f_\alpha''(x)-xf_\alpha'(x)+\alpha^2f_\alpha(x)=0,\\
\bigl(1-x^2\bigr)f_\alpha^{(3)}(x)-3xf_\alpha''(x)+\bigl(\alpha^2-1\bigr)f_\alpha'(x)=0,\\
\bigl(1-x^2\bigr)f_\alpha^{(4)}(x)-5xf_\alpha^{(3)}(x)+\bigl(\alpha^2-4\bigr)f_\alpha''(x)=0,
\end{gather*}
and, inductively,
\begin{equation}\label{n-deriv-gen+alpha}
\bigl(1-x^2\bigr)f_\alpha^{(k+2)}(x)-(2k+1)xf_\alpha^{(k+1)}(x)+\bigl(\alpha^2-k^2\bigr)f_\alpha^{(k)}(x)=0, \quad k\ge0.
\end{equation}
Letting $x\to0$ in~\eqref{n-deriv-gen+alpha} results in
\begin{equation}\label{n-deriv-x=0+alpha}
f_\alpha^{(k+2)}(0)+\bigl(\alpha^2-k^2\bigr)f_\alpha^{(k)}(0)=0, \quad k\ge0.
\end{equation}
Recusing the relation~\eqref{n-deriv-x=0+alpha} and considering $f_\alpha(0)=1$ and $f_\alpha'(0)=0$ arrive at
\begin{equation*}
f^{(2k)}(0)=\prod_{\ell=1}^{k}\bigl[4(\ell-1)^2-\alpha^2\bigr]
\quad\text{and}\quad
f^{(2k+1)}(0)=0
\end{equation*}
for $k\ge0$. Consequently, the series expansion~\eqref{cos-arcsin-ser} is valid.
\par
The series expansion~\eqref{sin-arcsin-ser} can be derived similarly.
\par
Let $f_\alpha(x)=\cos(\alpha\arccos x)$. Then the derivatives $f_\alpha^{(k)}(x)$ and $f_\alpha^{(k)}(0)$ satisfy
\begin{gather*}
f_\alpha'(x)=\frac{\alpha}{\sqrt{1-x^2}\,}\sin(\alpha\arccos x)
=\frac{\alpha}{\sqrt{1-x^2}\,}\sqrt{1-f_\alpha^2(x)}\,,\\
\bigl(1-x^2\bigr)\bigl[f_\alpha'(x)\bigr]^2-\alpha^2\bigl[1-f_\alpha^2(x)\bigr]=0,
\end{gather*}
and, inductively, the differential equation~\eqref{n-deriv-gen+alpha}. Setting $x\to1^-$ in~\eqref{n-deriv-gen+alpha} acquires
\begin{equation}\label{n-deriv-x=1+alpha}
(2k+1)f_\alpha^{(k+1)}(1)=\bigl(\alpha^2-k^2\bigr)f_\alpha^{(k)}(1), \quad k\ge0.
\end{equation}
Letting $x\to0$ in~\eqref{n-deriv-gen+alpha} leads to~\eqref{n-deriv-x=0+alpha}. Since $f_\alpha(1)=1$, $f_\alpha(0)=\cos\frac{\alpha\pi}{2}$, and $f_\alpha'(0)=\alpha\sin\frac{\alpha\pi}{2}$, from~\eqref{n-deriv-x=1+alpha} and~\eqref{n-deriv-x=0+alpha}, we obtain
\begin{equation*}
f_\alpha^{(k)}(1)=\prod_{\ell=0}^{k-1}\frac{\alpha^2-\ell^2}{2\ell+1}, \quad f_\alpha^{(2k)}(0)=\biggl(\cos\frac{\alpha\pi}{2}\biggr)\prod_{\ell=1}^{k}\bigl[4(\ell-1)^2-\alpha^2\bigr],
\end{equation*}
and
\begin{equation*}
f_\alpha^{(2k+1)}(0)=\alpha\biggl(\sin\frac{\alpha\pi}{2}\biggr)\prod_{\ell=1}^{k}\bigl[(2\ell-1)^2-\alpha^2\bigr]
\end{equation*}
for $k\ge0$. As a result, we acquire the series expansions~\eqref{cos-arccos-ser} and~\eqref{cos-arccos-ser-x=0}.
\par
Let $f_\alpha(x)=\sin(\alpha\arccos x)$. Then
\begin{gather*}
f_\alpha'(x)=-\frac{\alpha}{\sqrt{1-x^2}}\cos(\alpha\arccos x)
=-\frac{\alpha}{\sqrt{1-x^2}}\sqrt{1-f_\alpha^2(x)}\,,\\
\bigl(1-x^2\bigr)\bigl[f_\alpha'(x)\bigr]^2-\alpha^2\bigl[1-f_\alpha^2(x)\bigr]=0,\\
\bigl(1-x^2\bigr)f_\alpha''(x)-xf_\alpha'(x)+\alpha^2f_\alpha(x)=0,\\
\bigl(1-x^2\bigr)f_\alpha^{(3)}(x)-3xf_\alpha''(x)+\bigl(\alpha^2-1\bigr)f_\alpha'(x)=0,
\end{gather*}
and, inductively, the differential equation~\eqref{n-deriv-gen+alpha} and the recursive relation~\eqref{n-deriv-x=0+alpha} are valid. Employing the recursive relation~\eqref{n-deriv-x=0+alpha} and using
\begin{equation*}
f_\alpha(0)=\sin\frac{\alpha\pi}{2} \quad\text{and}\quad f_\alpha'(0)=-\alpha\cos\frac{\alpha\pi}{2}
\end{equation*}
results in
\begin{equation*}
f_\alpha^{(2k)}(0)=\biggl(\sin\frac{\alpha\pi}{2}\biggr) \prod_{\ell=1}^{k}\bigl[4(\ell-1)^2-\alpha^2\bigr]
\end{equation*}
and
\begin{equation*}
f_\alpha^{(2k+1)}(0)=-\alpha\biggl(\cos\frac{\alpha\pi}{2}\biggr) \prod_{\ell=1}^{k}\bigl[(2\ell-1)^2-\alpha^2\bigr]
\end{equation*}
for $k\ge0$. Accordingly, the series expansion~\eqref{sin-arccos-ser} follows.
The proof of Lemma~\ref{cosh-sinh-arcsin-arccos-thm} is complete.
\end{proof}

\begin{rem}
In the paper~\cite{Qureshi-Majid-Bhat-Aeajmms-2020}, among other things, three authors established series expansions at $x=0$ or $x=1$ of the functions
\begin{gather*}
\exp(\alpha\arccos x),\quad
\frac{\exp(\alpha\arccos x)}{\sqrt{1-x^2}\,},\quad
\frac{\arccos x}{\sqrt{1-x^2}\,},\\
\frac{\sin(\alpha\arccos x)}{\sqrt{1-x^2}\,},\quad
\exp(\alpha\arccosh x),\quad
\frac{\sin(\alpha\arccosh x)}{\sqrt{x^2-1}\,}, \\
\frac{\sinh(\alpha\arccosh x)}{\sqrt{x^2-1}\,},\quad
\cos(\alpha\arccos x),\quad
\cosh(\alpha\arccos x), \\
\cos(\alpha\arccosh x),\quad
\cosh(\alpha\arccosh x),\quad
\sin(\alpha\arccos x), \\
\sinh(\alpha\arccos x),\quad
\frac{\cos(\alpha\arccos x)}{\sqrt{1-x^2}\,},\quad
\frac{\cosh(\alpha\arccos x)}{\sqrt{1-x^2}\,}, \\
\frac{\cos(\alpha\arccosh x)}{\sqrt{1-x^2}\,},\quad
\frac{\sinh(\alpha\arccos x)}{\sqrt{1-x^2}\,},\quad
\frac{\cosh(\alpha\arccosh x)}{\sqrt{1-x^2}\,}
\end{gather*}
for $\alpha\in\mathbb{C}\setminus\{0\}$ in terms of the Gauss hypergeometric function ${}_2F_1(\alpha,\beta;\gamma;z)$ which can be defined~\cite[Section~5.9]{Temme-96-book} by
\begin{equation*}
{}_2F_1(\alpha,\beta;\gamma;z)=\sum_{k=0}^{\infty}\frac{(\alpha)_k(\beta)_k}{(\gamma)_k}\frac{z^k}{k!}, \quad |z|<1
\end{equation*}
for complex numbers $\alpha,\beta\in\mathbb{C}\setminus\{0\}$ and $\gamma\in\mathbb{C}\setminus\{0,-1,-2,\dotsc\}$, where $(\alpha)_k$, $(\beta)_k$, and $(\gamma)_k$ are defined by~\eqref{rising-Factorial} or~\eqref{extended-Pochhammer-dfn}.
\par
By the way, we point out that the series expansions~(2.1), (4.1), (4.3), (4.4), (4.7), and~(4.9) in the paper~\cite{Qureshi-Majid-Bhat-Aeajmms-2020} should be wrong.
\end{rem}

\begin{lem}\label{Partial-Bell-Conclusions-Lem}
Let $n\ge k\ge0$ and $\alpha,\beta\in\mathbb{C}$. Then
\begin{enumerate}
\item
The Fa\`a di Bruno formula can be described in terms of partial Bell polynomials $\bell_{n,k}$ by
\begin{equation}\label{Bruno-Bell-Polynomial}
\frac{\td^n}{\td t^n}f\circ h(t)
=\sum_{k=0}^nf^{(k)}(h(t)) \bell_{n,k}\bigl(h'(t),h''(t),\dotsc,h^{(n-k+1)}(t)\bigr), \quad n\in\mathbb{N}_0.
\end{equation}
\item
Partial Bell polynomials $\bell_{n,k}$ satisfy the identities
\begin{align}\label{Bell(n-k)}
\bell_{n,k}\bigl(\alpha\beta x_1,\alpha \beta^2x_2,\dotsc,\alpha \beta^{n-k+1}x_{n-k+1}\bigr)
&=\alpha^k\beta^n\bell_{n,k}(x_1,x_2,\dotsc,x_{n-k+1}),\\
\label{Bell-x-1-0-eq}
\bell_{n,k}(\alpha,1,0,\dotsc,0)
&=\frac{(n-k)!}{2^{n-k}}\binom{n}{k}\binom{k}{n-k}\alpha^{2k-n},\\
\label{Bell-formula-S}
\bell_{n,k}((-1)!!,1!!,3!!,\dotsc,[2(n-k)-1]!!)
&=[2(n-k)-1]!!\binom{2n-k-1}{2(n-k)},
\end{align}
and
\begin{equation}\label{113-final-formula}
\frac1{k!}\Biggl(\sum_{m=1}^{\infty} x_m\frac{t^m}{m!}\Biggr)^k =\sum_{n=k}^{\infty} \bell_{n,k}(x_1,x_2,\dotsc,x_{n-k+1})\frac{t^n}{n!}.
\end{equation}
\end{enumerate}
\end{lem}

These formulas in Lemma~\ref{Partial-Bell-Conclusions-Lem} can be found in~\cite[p.~412]{Charalambides-book-2002}, \cite[pp.~134--135 and p.~139]{Comtet-Combinatorics-74}, \cite[Theorem~4.1]{Spec-Bell2Euler-S.tex}, \cite[p.~169, (3.6)]{CDM-68111.tex}, and~\cite[Theorem~1.2]{JAAC961.tex}, respectively. These identities in Lemma~\ref{Partial-Bell-Conclusions-Lem} can also be found in the survey and review article~\cite{Bell-value-elem-funct.tex}.

\section{Taylor's series expansions of $\bigl[\frac{(\arccos x)^{2}}{2(1-x)}\bigr]^k$ and $\bigl[\frac{(\arccosh x)^{2}}{2(1-x)}\bigr]^k$}

In this section, by virtue of some conclusions in Lemmas~\ref{square-sum-prod-lem}, \ref{square-sum-prod-odd-lem}, and~\ref{cosh-sinh-arcsin-arccos-thm}, we establish Taylor's series expansions around $x=1^-$ of the functions $\bigl[\frac{(\arccos x)^{2}}{2(1-x)}\bigr]^k$ and $\bigl[\frac{(\arccosh x)^{2}}{2(1-x)}\bigr]^k$ in terms of $Q(k,m)$ defined by the formula~\eqref{Q(m-k)-sum-dfn}.

\begin{thm}\label{arccos-ser-expan-thm}
For $k\in\mathbb{N}$ and $|x|<1$, we have
\begin{equation}\label{arccos-2(1-x)-power-ser-expan}
\biggl[\frac{(\arccos x)^{2}}{2(1-x)}\biggr]^k
=1+(2k)!\sum_{n=1}^{\infty} \frac{Q(2k,2n)}{(2k+2n)!}[2(x-1)]^{n}
\end{equation}
and
\begin{equation}\label{arccosh-square-ser}
\biggl[\frac{(\arccosh x)^{2}}{2(x-1)}\biggr]^k
=1+(2k)!\sum_{n=1}^{\infty} \frac{Q(2k,2n)}{(2k+2n)!}[2(x-1)]^{n},
\end{equation}
where $Q(2k,2n)$ is defined by~\eqref{Q(m-k)-sum-dfn}.
\end{thm}

\begin{proof}
Replacing $\alpha$ by $\ti\alpha$ in~\eqref{square-sum-first-stirl-eq} leads to
\begin{equation}\label{squ-sirst-stirl-eq}
\prod_{\ell=1}^{k}\bigl(\ell^2-\alpha^2\bigr)
=(-1)^{k}\sum_{j=0}^{k}\Biggl[\sum_{\ell=2j+1}^{2k+1}\binom{\ell}{2j+1}s(2k+1,\ell)k^{\ell-2j-1}\Biggr]\alpha^{2j}.
\end{equation}
From Taylor's series expansion~\eqref{cos-arccos-ser} and the identity~\eqref{squ-sirst-stirl-eq}, it follows that
\begin{gather*}
\sum_{k=0}^{\infty}(-1)^k\frac{(\alpha\arccos x)^{2k}}{(2k)!}
=1+(x-1)\alpha^2-\alpha^2\sum_{k=2}^{\infty} \frac{(-1)^k}{(2k-1)!!} \Biggl[\prod_{\ell=1}^{k-1}\bigl(\ell^2-\alpha^2\bigr)\Biggr] \frac{(x-1)^k}{k!}\\
=1+(x-1)\alpha^2+\alpha^2\sum_{k=2}^{\infty} \frac{1}{(2k-1)!!} \Biggl[\sum_{j=0}^{k-1} \alpha^{2j}\sum_{\ell=2j+1}^{2k-1}\binom{\ell}{2j+1}s(2k-1,\ell)(k-1)^{\ell-2j-1}\Biggr] \frac{(x-1)^{k}}{k!}\\
=1+(x-1)\alpha^2+\alpha^2\sum_{k=2}^{\infty} \frac{1}{(2k-1)!!} \Biggl[\sum_{\ell=1}^{2k-1}\ell s(2k-1,\ell)(k-1)^{\ell-1}\Biggr] \frac{(x-1)^{k}}{k!}\\
+\alpha^2\sum_{k=2}^{\infty} \frac{1}{(2k-1)!!}\frac{(x-1)^{k}}{k!} \sum_{j=2}^{k} \alpha^{2j-2} \sum_{\ell=2j-1}^{2k-1} \binom{\ell}{2j-1}s(2k-1,\ell)(k-1)^{\ell-2j+1}\\
=1+(x-1)\alpha^2+\alpha^2\sum_{k=2}^{\infty} \frac{1}{(2k-1)!!} \Biggl[\sum_{\ell=1}^{2k-1}\ell s(2k-1,\ell)(k-1)^{\ell-1}\Biggr] \frac{(x-1)^{k}}{k!}\\
+\sum_{k=2}^{\infty}\Biggl[\sum_{m=k}^{\infty} \frac{1}{(2m-1)!!}\frac{(x-1)^{m}}{m!} \sum_{\ell=2k-1}^{2m-1} \binom{\ell}{2k-1}s(2m-1,\ell)(m-1)^{\ell-2k+1}\Biggr]\alpha^{2k}.
\end{gather*}
This means that
\begin{align*}
-\frac{(\arccos x)^{2}}{2!}
&=x-1+\sum_{k=2}^{\infty} \frac{1}{(2k-1)!!} \Biggl[\sum_{\ell=1}^{2k-1}\ell s(2k-1,\ell)(k-1)^{\ell-1}\Biggr] \frac{(x-1)^{k}}{k!}\\
&=x-1+\sum_{k=1}^{\infty} \Biggl[\sum_{\ell=0}^{2k}(\ell+1) s(2k+1,\ell+1)k^{\ell}\Biggr] \frac{[2(x-1)]^{k+1}}{(2k+2)!},
\end{align*}
where we used the identity~\eqref{square-sum-first-alpha=0} in Lemma~\ref{square-sum-prod-lem} or the identity~\eqref{square-sum-first-alpha=0-Q},
and that
\begin{align*}
(-1)^k\frac{(\arccos x)^{2k}}{(2k)!}
&=\sum_{m=k}^{\infty} \Biggl[\sum_{\ell=2k-1}^{2m-1} \binom{\ell}{2k-1}s(2m-1,\ell)(m-1)^{\ell-2k+1}\Biggr]\frac{(x-1)^{m}}{(2m-1)!!m!}\\
&=\sum_{m=0}^{\infty} Q(2k,2m) \frac{[2(x-1)]^{m+k}}{(2k+2m)!}
\end{align*}
for $k\ge2$. Consequently, the series expansions
\begin{equation}\label{arccos-square-ser}
\frac{(\arccos x)^{2}}{2!}
=\sum_{m=0}^{\infty} \frac{m!}{(2m+1)!!} \frac{(1-x)^{m+1}}{m+1}, \quad |x|<1
\end{equation}
and
\begin{equation}\label{arccos-pow-ser}
\frac{(\arccos x)^{2k}}{(2k)!}
=\sum_{m=0}^{\infty} (-1)^{m}Q(2k,2m) \frac{[2(1-x)]^{m+k}}{(2k+2m)!}
\end{equation}
for $k\ge2$ and $|x|<1$ are valid.
\par
By similar arguments as done above, from the series expansion~\eqref{cosh-arccos-ser}, we can also recover the series expansion~\eqref{arccos-square-ser} and~\eqref{arccos-pow-ser}.
\par
The series expansions~\eqref{arccos-square-ser} and~\eqref{arccos-pow-ser} can be reformulated as
\begin{equation*}
\frac{(\arccos x)^{2}}{2(1-x)}
=1+\sum_{m=1}^{\infty} \frac{(m!)^2}{(2m+1)!!(m+1)} \frac{(1-x)^{m}}{m!}
\end{equation*}
and
\begin{gather*}
\biggl[\frac{(\arccos x)^{2}}{2(1-x)}\biggr]^k
=1+\sum_{m=1}^{\infty} \frac{(-1)^{m}}{(2m-1)!!\binom{2m+2k}{2k}} Q(2k,2m)\frac{(1-x)^{m}}{m!}
\end{gather*}
for $k\ge2$. Making use of the identity~\eqref{square-sum-first-alpha=0} or~\eqref{square-sum-first-alpha=0-Q}, we obtain
\begin{equation*}
\frac{(-1)^{m}}{(2m-1)!!\binom{2m+2}{2}} Q(2,2m)
=\frac{(m!)^2}{(2m+1)!!(m+1)}, \quad m\in\mathbb{N}.
\end{equation*}
Consequently, the series expansions~\eqref{arccos-square-ser} and~\eqref{arccos-pow-ser} can be unified as the series expansion~\eqref{arccos-2(1-x)-power-ser-expan}.
\par
By the relation
\begin{equation}\label{arccos-arccosh-relation}
\arccos x=-\ti\arccosh x,
\end{equation}
from~\eqref{arccos-2(1-x)-power-ser-expan}, we deduce~\eqref{arccosh-square-ser}.
The proof of Theorem~\ref{arccos-ser-expan-thm} is complete.
\end{proof}

\begin{cor}\label{Minus-sh-square-ser-COR}
For $k\in\mathbb{N}$ and $|x|<1$, we have
\begin{equation}\label{Minus-2(1-x)-power-ser-expan}
\biggl[\frac{(\pi-\arccos x)^{2}}{2(1+x)}\biggr]^k
=1+(2k)!\sum_{m=1}^{\infty} (-1)^m\frac{Q(2k,2m)} {(2k+2m)!}[2(x+1)]^{m}
\end{equation}
and
\begin{equation}\label{Minus-sh-square-ser}
(-1)^k\biggl[\frac{(\pi+\ti\arccosh x)^{2}}{2(1+x)}\biggr]^k
=1+\sum_{m=1}^{\infty} (-1)^m \frac{(2k)!(2m)!!}{(2k+2m)!}Q(2k,2m)\frac{(x+1)^{m}}{m!},
\end{equation}
where $Q(2k,2m)$ is defined by~\eqref{Q(m-k)-sum-dfn}.
\end{cor}

\begin{proof}
This follows from replacing $x$ by $-x$ in~\eqref{arccos-2(1-x)-power-ser-expan} and~\eqref{arccosh-square-ser} and utilizing the relations
\begin{equation*}
\arccos x+\arccos(-x)=\pi
\end{equation*}
and~\eqref{arccos-arccosh-relation}.
\end{proof}

\begin{cor}\label{series-approx-pi-first2cor}
For $k\in\mathbb{N}$, we have
\begin{equation}\label{Pi-power-ser-repres}
\biggl(\frac{\pi^{2}}{8}\biggr)^k
=1+(2k)!\sum_{m=1}^{\infty} (-1)^{m}2^m\frac{Q(2k,2m)}{(2k+2m)!}.
\end{equation}
In particular, we have
\begin{equation}\label{pi-square-div-8}
\frac{\pi^2}{8}=\sum_{m=1}^{\infty}\frac{2^{m}}{m^2}\frac{1}{\binom{2m}{m}}.
\end{equation}
\end{cor}

\begin{proof}
The series representation~\eqref{Pi-power-ser-repres} of $\pi^{2k}$ follows from letting $x=0$ in either~\eqref{arccos-2(1-x)-power-ser-expan}, \eqref{arccosh-square-ser}, \eqref{Minus-2(1-x)-power-ser-expan}, or~\eqref{Minus-sh-square-ser}.
\par
The series representation~\eqref{pi-square-div-8} follows from taking $k=1$ in~\eqref{Pi-power-ser-repres}, or setting $k=2$ in~\eqref{arcsinh-series-pi}, and then making use of the identity~\eqref{square-sum-first-alpha=0} or~\eqref{square-sum-first-alpha=0-Q}.
\end{proof}

\begin{rem}\label{Rem4.1-Pi-Speed}
The series representation~\eqref{pi-square-div-8} recovers a conclusion in~\cite[Theorem~5.1]{series-arccos-v2.tex}.
\par
As for series representations of $\pi^2$, in~\cite[p.~453, (14)]{Lehmer-Monthly-1985} and the paper~\cite{Zeta-luo.tex}, among other things, we find out
\begin{equation}\label{four-ser-represen-pi-square}
\begin{aligned}
\frac{\pi^2}{6}&=\sum_{m=0}^{\infty}\frac{1}{(m+1)^2}, &
\frac{\pi^2}{8}&=\sum_{m=0}^{\infty}\frac{1}{(2m+1)^2},\\
\frac{\pi^2}{12}&=\sum_{m=0}^{\infty}\frac{(-1)^m}{(m+1)^2}, &
\frac{\pi^2}{18}&=\sum_{m=1}^{\infty}\frac{1}{m^2}\frac{1}{\binom{2m}{m}},
\end{aligned}
\end{equation}
which are different from~\eqref{pi-square-div-8}. The last series representation in~\eqref{four-ser-represen-pi-square} is derived from letting $k=2$ and $x=\frac{1}{2}$ in~\eqref{arcsin-series-expansion-unify}.
\par
Because
\begin{equation}
\begin{gathered}\label{five-limits-hypergeometric}
\lim_{m\to\infty}\biggl[\frac{2^{m}}{m^2}\frac{8}{\binom{2m}{m}}\biggr]^{1/m}=\frac{1}{2}, \quad
\lim_{m\to\infty}\biggl[\frac{6}{(m+1)^2}\biggr]^{1/m}=1, \quad
\lim_{m\to\infty}\biggl[\frac{8}{(2m+1)^2}\biggr]^{1/m}=1, \\
\lim_{m\to\infty}\biggl[\frac{(-1)^m12}{(m+1)^2}\biggr]^{1/m}=1, \quad
\lim_{m\to\infty}\biggl[\frac{18}{m^2}\frac{1}{\binom{2m}{m}}\biggr]^{1/m}=\frac{1}{4},
\end{gathered}
\end{equation}
we regard that the last series representation in~\eqref{four-ser-represen-pi-square} in~\cite[p.~453, (14)]{Lehmer-Monthly-1985} converges to $\pi^2$ quicker than~\eqref{pi-square-div-8} and other three in~\eqref{four-ser-represen-pi-square}.
The first unsolved problem posed on December 13, 2010 by Herbert S. Wilf (1931–2012) is about the convergent speed of rational approximations of the circular constant $\pi$. For more details on this unsolved problem, see~\cite[Remark~7.7]{series-arccos-v2.tex}.
\par
For $k\in\mathbb{N}$, let
\begin{equation*}
L(k)=\lim_{m\to\infty}\biggl[(2k)!(-1)^{m}2^m\frac{Q(2k,2m)}{(2k+2m)!}\biggr]^{1/m}
=2\lim_{m\to\infty}\biggl[(-1)^{m}\frac{Q(2k,2m)}{(2k+2m)!}\biggr]^{1/m}.
\end{equation*}
The first limit in~\eqref{five-limits-hypergeometric} means that $L(1)=\frac{1}{2}$.
What is the convergent speed of the hypergeometric term in~\eqref{Pi-power-ser-repres} for $k\ge2$? Equivalently speaking, what is the limit $L(k)$ for $k\ge2$? Is the limit $L(k)$ a decreasing sequence in $k\ge2$? What is the limit $\lim_{k\to\infty}[L(k)]^{1/k}$?
\end{rem}

\begin{cor}\label{COR-4-Deriv}
For $k,m\in\mathbb{N}$, we have
\begin{align*}
\Biggl(\biggl[\frac{(\arccos x)^{2}}{2(1-x)}\biggr]^k\Biggr)^{(m)}\Bigg|_{x=1^-}
&=\frac{(2k)!(2m)!!}{(2k+2m)!}Q(2k,2m),\\
\Biggl(\biggl[\frac{(\arccosh x)^{2}}{2(1-x)}\biggr]^k\Biggr)^{(m)}\Bigg|_{x=1^-}
&=(-1)^k\frac{(2k)!(2m)!!}{(2k+2m)!}Q(2k,2m),\\
\Biggl(\biggl[\frac{(\pi-\arccos x)^{2}}{2(1+x)}\biggr]^k\Biggr)^{(m)}\Bigg|_{x=(-1)^+}
&=(-1)^m\frac{(2k)!(2m)!!}{(2k+2m)!}Q(2k,2m),
\end{align*}
and
\begin{equation*}
\Biggl(\biggl[\frac{(\pi+\ti\arccosh x)^{2}}{2(1+x)}\biggr]^k\Biggr)^{(m)}\Bigg|_{x=(-1)^+}
=(-1)^{k+m}\frac{(2k)!(2m)!!}{(2k+2m)!}Q(2k,2m),
\end{equation*}
where $Q(2k,2m)$ is defined by~\eqref{Q(m-k)-sum-dfn}.
\end{cor}

\begin{proof}
These left-hand and right-hand derivatives follow from Taylor's series expansions~\eqref{arccos-2(1-x)-power-ser-expan}, \eqref{arccosh-square-ser}, \eqref{Minus-2(1-x)-power-ser-expan}, and~\eqref{Minus-sh-square-ser}.
\end{proof}

\begin{cor}\label{arccosh-squ-deriv-form2-cor}
For $k,n\in\mathbb{N}$, we have
\begin{equation}\label{arccos-square-deriv-formula}
\bigl[(\arccos x)^{2k}\bigr]^{(n)}\Big|_{x=1^-}
=
\begin{dcases}
0, & n<k;\\
(-1)^k(2k)!!, & n=k;\\
(-1)^k\frac{(2k)!}{(2n-1)!!}Q(2k,2n-2k), & n>k
\end{dcases}
\end{equation}
and
\begin{equation}\label{arccosh-square-deriv-formula}
\bigl[(\arccosh x)^{2k}\bigr]^{(n)}\Big|_{x=1^-}
=
\begin{dcases}
0, & n<k;\\
(2k)!!, & n=k;\\
\frac{(2k)!}{(2n-1)!!}Q(2k,2n-2k), & n>k.
\end{dcases}
\end{equation}
\end{cor}

\begin{proof}
For $k\in\mathbb{N}$ and $|x|<1$, the series expansions~\eqref{arccos-2(1-x)-power-ser-expan} and~\eqref{arccosh-square-ser} can be reformulated as
\begin{equation}\label{arccos-square-x=1-ser}
(\arccos x)^{2k}
=(-1)^k(2k)!!\frac{(x-1)^k}{k!}+(-1)^k(2k)!\sum_{m=1}^{\infty} \frac{Q(2k,2m)}{(2k+2m-1)!!}\frac{(x-1)^{k+m}}{(k+m)!}
\end{equation}
and
\begin{equation}\label{arccosh-square-x=1-ser}
(\arccosh x)^{2k}
=(2k)!!\frac{(x-1)^k}{k!}+(2k)!\sum_{m=1}^{\infty} \frac{Q(2k,2m)}{(2k+2m-1)!!}\frac{(x-1)^{k+m}}{(k+m)!}.
\end{equation}
These forms of series expansions~\eqref{arccos-square-x=1-ser} and~\eqref{arccosh-square-x=1-ser} imply the formulas in~\eqref{arccos-square-deriv-formula} and~\eqref{arccosh-square-deriv-formula}.
\end{proof}

\begin{cor}\label{arccosH-square-ser-expan-thm}
For $k\in\mathbb{N}$ and $|x|<1$, we have
\begin{equation}
\begin{aligned}\label{arccos-square-ser-expan}
(\arccos x)^{2k}
&=\sum_{j=0}^{k}(-1)^{j}\Biggl[2^k \binom{k}{j} +(2k)!\sum_{m=1}^{\infty} \frac{(-1)^{m}}{(k+m)!} \frac{Q(2k,2m)}{(2k+2m-1)!!} \binom{k+m}{j}\Biggr]x^j\\
&\quad+(-1)^{k}(2k)!\sum_{j=k+1}^{\infty}(-1)^{j}\Biggl[\sum_{m=0}^{\infty}\frac{(-1)^{j+m}}{(j+m)!} \frac{Q(2k,2j+2m-2k)}{(2j-1)!!} \binom{j+m}{j}\Biggr]x^j
\end{aligned}
\end{equation}
and
\begin{equation}
\begin{aligned}\label{arccosh-square-ser-expan}
(\arccos x)^{2k}
&=(-1)^{k}\sum_{j=0}^{k}(-1)^{j}\Biggl[2^k \binom{k}{j} +(2k)!\sum_{m=1}^{\infty} \frac{(-1)^{m}}{(k+m)!} \frac{Q(2k,2m)}{(2k+2m-1)!!} \binom{k+m}{j}\Biggr]x^j\\
&\quad+(2k)!\sum_{j=k+1}^{\infty}(-1)^{j}\Biggl[\sum_{m=0}^{\infty}\frac{(-1)^{j+m}}{(j+m)!} \frac{Q(2k,2j+2m-2k)}{(2j-1)!!} \binom{j+m}{j}\Biggr]x^j,
\end{aligned}
\end{equation}
where $Q(2k,2m)$ is given by~\eqref{Q(m-k)-sum-dfn}.
\end{cor}

\begin{proof}
For $k\in\mathbb{N}$ and $|x|<1$, by the binomial theorem, the series expansion~\eqref{arccos-square-x=1-ser} can be rewritten as
\begin{gather*}
(\arccos x)^{2k}
=2^k \sum_{j=0}^{k}(-1)^{j}\binom{k}{j}x^j +(2k)!\sum_{m=1}^{\infty} \sum_{j=0}^{k+m}\frac{(-1)^{m-j}}{(k+m)!} \frac{Q(2k,2m)}{(2k+2m-1)!!} \binom{k+m}{j}x^j\\
=2^k \sum_{j=0}^{k}(-1)^{j}\binom{k}{j}x^j +(2k)!\Biggl(\sum_{j=0}^{k}\sum_{m=1}^{\infty} +\sum_{j=k+1}^{\infty}\sum_{m=j-k}^{\infty}\Biggr) \frac{(-1)^{m-j}}{(k+m)!} \frac{Q(2k,2m)}{(2k+2m-1)!!} \binom{k+m}{j}x^j\\
=\sum_{j=0}^{k}(-1)^{j}\Biggl[2^k \binom{k}{j} +(2k)!\sum_{m=1}^{\infty} \frac{(-1)^{m}}{(k+m)!} \frac{Q(2k,2m)}{(2k+2m-1)!!} \binom{k+m}{j}\Biggr]x^j\\
+(2k)!\sum_{j=k+1}^{\infty}(-1)^{j}\Biggl[\sum_{m=j-k}^{\infty}\frac{(-1)^{m}}{(k+m)!} \frac{Q(2k,2m)}{(2k+2m-1)!!} \binom{k+m}{j}\Biggr]x^j.
\end{gather*}
The series expansion~\eqref{arccos-square-ser-expan} is thus proved.
\par
Similarly, from~\eqref{arccosh-square-x=1-ser}, we conclude~\eqref{arccosh-square-ser-expan}. Theorem~\ref{arccosH-square-ser-expan-thm} is thus proved.
\end{proof}

\begin{rem}
What are closed-form expressions of coefficients in Maclaurin's series expansion around the point $x=0$ of the real power $(\arccos x)^\alpha$ for $\alpha\in\mathbb{R}$ and $|x|<1$? For answers to this question, please refer to~\cite[Section~3]{series-arccos-v2.tex}.
\end{rem}

\begin{cor}\label{arccos-odd-power-deriv-thm}
For $m,n\in\mathbb{N}$, we have
\begin{equation}\label{arccos-odd-deriv-eq}
\bigl[(\arccos x)^{2n-1}\bigr]^{(m)}\Big|_{x=1^-}
=\begin{dcases}
0, & m<n;\\
\infty, & m\ge n
\end{dcases}
\end{equation}
and
\begin{equation}\label{arccosh-odd-deriv-eq}
\bigl[(\arccosh x)^{2n-1}\bigr]^{(m)}\Big|_{x=1^-}
=\begin{dcases}
0, & m<n;\\
\infty, & m\ge n.
\end{dcases}
\end{equation}
Consequently, the functions $(\arccos x)^{2n-1}$ and $(\arccosh x)^{2n-1}$ for $n\in\mathbb{N}$ cannot be expanded into Taylor's series expansions at the point $x=1^-$.
\end{cor}

\begin{proof}
It is easy to verify that
\begin{equation}\label{arccos-power-sim0relation-x=1}
(\arccos x)^{2}\sim 2(1-x), \quad x\to1^-.
\end{equation}
\par
By the Fa\`a di Bruno formula~\eqref{Bruno-Bell-Polynomial} and in the light of the identities~\eqref{Bell(n-k)} and~\eqref{Bell-x-1-0-eq}, we obtain
\begin{align*}
\biggl(\frac{1}{\sqrt{1-x^2}\,}\biggr)^{(m)}
&=\sum_{j=0}^{m}\biggl\langle-\frac{1}{2}\biggr\rangle_{j}\bigl(1-x^2\bigr)^{-1/2-j} \bell_{m,j}\bigl(-2x,-2,0,\dotsc,0\bigr)\\
&=\sum_{j=0}^{m}(-1)^j\frac{(2j-1)!!}{2^j}\frac{(-2)^j}{(1-x^2)^{j+1/2}} \bell_{m,j}\bigl(x,1,0,\dotsc,0\bigr)\\
&=\sum_{j=0}^{m} \frac{(2j-1)!!}{(1-x^2)^{j+1/2}} \frac{(m-j)!}{2^{m-j}}\binom{m}{j}\binom{j}{m-j}x^{2j-m}\\
&=\sum_{j=0}^{m} \frac{(2j-1)!!}{(1+x)^{j+1/2}} \frac{(m-j)!}{2^{m-j}}\binom{m}{j}\binom{j}{m-j}\frac{x^{2j-m}}{(1-x)^{j+1/2}}
\end{align*}
for $m\in\mathbb{N}_0$. This implies that
\begin{equation}\label{arcccos-sin=deriv-sim-x=1}
\biggl(\frac{1}{\sqrt{1-x^2}\,}\biggr)^{(m)}\sim \frac{(2m-1)!!}{[2(1-x)]^{m+1/2}}, \quad x\to1^-, \quad m\in\mathbb{N}_0.
\end{equation}
\par
Utilizing the Fa\`a di Bruno formula~\eqref{Bruno-Bell-Polynomial}, employing the identities~\eqref{Bell(n-k)} and~\eqref{Bell-formula-S}, and making use of~\eqref{arccos-power-sim0relation-x=1} and~\eqref{arcccos-sin=deriv-sim-x=1}, we acquire
\begin{gather*}
\bigl[(\arccos x)^{2n-1}\bigr]^{(m)}
=\sum_{j=0}^{m}\langle2n-1\rangle_{j}(\arccos x)^{2n-j-1} \\
\times\bell_{m,j}\biggl(-\frac{1}{\sqrt{1-x^2}\,}, \biggl(-\frac{1}{\sqrt{1-x^2}\,}\biggr)', \dotsc, \biggl(-\frac{1}{\sqrt{1-x^2}\,}\biggr)^{(m-j)}\biggr)\\
\sim\sum_{j=0}^{m}\langle2n-1\rangle_{j}[2(1-x)]^{n-(j+1)/2} (-1)^j\\
\times\bell_{m,j}\biggl(\frac{(-1)!!}{[2(1-x)]^{1/2}}, \frac{1!!}{[2(1-x)]^{3/2}}, \frac{3!!}{[2(1-x)]^{5/2}},\dotsc, \frac{[2(m-j)-1]!!}{[2(1-x)]^{m-j+1/2}}\biggr)\\
=\sum_{j=0}^{m}(-1)^j\langle2n-1\rangle_{j}[2(1-x)]^{n-(j+1)/2} \frac{[2(1-x)]^{j/2}}{[2(1-x)]^{m}}\bell_{m,j}((-1)!!,1!!, 3!!, \dotsc, [2(m-j)-1]!!)\\
=[2(1-x)]^{n-m-1/2} \sum_{j=0}^{m}(-1)^j\langle2n-1\rangle_{j} [2(m-j)-1]!!\binom{2m-j-1}{2(m-j)}\\
\to
\begin{dcases}
0, & n>m\\
\infty, & n\le m
\end{dcases}
\end{gather*}
as $x\to1^-$ for $m,n\in\mathbb{N}$. The results in~\eqref{arccos-odd-deriv-eq} are thus proved.
\par
Substituting~\eqref{arccos-arccosh-relation} into~\eqref{arccos-odd-deriv-eq} leads to~\eqref{arccosh-odd-deriv-eq}. Theorem~\ref{arccos-odd-power-deriv-thm} is therefore proved.
\end{proof}

\begin{cor}\label{comb-id-arccos-ven=pow-cor}
For $k\in\mathbb{N}_0$, we have
\begin{equation}\label{comb-id-arccos-ven=pow}
\sum_{j=0}^{k}(-1)^j\langle2k\rangle_{j} [2(k-j)-1]!!\binom{2k-j-1}{2(k-j)}
=(-1)^k(2k)!!.
\end{equation}
\end{cor}

\begin{proof}
As done in the proof of Theorem~\ref{arccos-odd-power-deriv-thm}, we arrive at
\begin{gather*}
\begin{aligned}
\bigl[(\arccos x)^{2k}\bigr]^{(n)}
&=\sum_{j=0}^{n}\langle2k\rangle_{j}(\arccos x)^{2k-j} \bell_{n,j}\biggl(-\frac{1}{\sqrt{1-x^2}\,}, \biggl(-\frac{1}{\sqrt{1-x^2}\,}\biggr)', \dotsc, \biggl(-\frac{1}{\sqrt{1-x^2}\,}\biggr)^{(n-j)}\biggr)
\end{aligned}\\
\sim\sum_{j=0}^{n}\langle2k\rangle_{j}[2(1-x)]^{k-j/2} (-1)^j\bell_{n,j}\biggl(\frac{(-1)!!}{[2(1-x)]^{1/2}}, \frac{1!!}{[2(1-x)]^{3/2}}, \dotsc, \frac{[2(n-j)-1]!!}{[2(1-x)]^{n-j+1/2}}\biggr)\\
=\sum_{j=0}^{n}(-1)^j\langle2k\rangle_{j}[2(1-x)]^{k-j/2} \frac{[2(1-x)]^{j/2}}{[2(1-x)]^{n}}\bell_{n,j}((-1)!!,1!!, 3!!, \dotsc, [2(n-j)-1]!!)\\
=[2(1-x)]^{k-n} \sum_{j=0}^{n}(-1)^j\langle2k\rangle_{j} [2(n-j)-1]!!\binom{2n-j-1}{2(n-j)}\\
\to
\begin{dcases}
0, & k>n\\
\sum_{j=0}^{k}(-1)^j\langle2k\rangle_{j} [2(k-j)-1]!!\binom{2k-j-1}{2(k-j)}, & k=n
\end{dcases}
\end{gather*}
as $x\to1^-$ for $k,n\in\mathbb{N}$. Comparing this result with~\eqref{arccos-square-deriv-formula} gives~\eqref{comb-id-arccos-ven=pow}.
\end{proof}

\begin{rem}
The identity~\eqref{comb-id-arccos-ven=pow} is similar to
\begin{equation*}
\sum_{k=0}^{n}k![2(n-k)-1]!! \binom{2n-k-1}{2(n-k)}=(2n-1)!!, \quad n\in\mathbb{N}_0,
\end{equation*}
which was respectively established in~\cite[p.10, (3.12)]{Wilf-Ward-2010P.tex} and in an unpublished paper titled ``Partial Bell polynomials, falling and rising factorials, Stirling numbers, and combinatorial identities''.
\end{rem}

\section{Taylor's series expansions of $\bigl[\frac{(\arccos x)^{2}}{2(1-x)}\bigr]^\alpha$}

In this section, via establishing a closed-form expression for the specific partial Bell polynomials at a sequence of the derivatives at $x=1^-$ of the function $\bigl[\frac{(\arccos x)^{2}}{2(1-x)}\bigr]^\alpha$, we present Taylor's series expansion at $x=1^-$ of the function $\bigl[\frac{(\arccos x)^{2}}{2(1-x)}\bigr]^\alpha$ for $\alpha\in\mathbb{R}$.

\begin{thm}\label{arccos-squar-(1-x)-bell-thm}
For $m\ge k\in\mathbb{N}$, we have
\begin{equation}
\begin{gathered}\label{arccos-squar-(1-x)-bell-Eq}
\bell_{m,k}\biggl(\biggl[\frac{(\arccos x)^{2}}{2(1-x)}\biggr]'\bigg|_{x=1^-}, \biggl[\frac{(\arccos x)^{2}}{2(1-x)}\biggr]''\bigg|_{x=1^-}, \dotsc, \biggl[\frac{(\arccos x)^{2}}{2(1-x)}\biggr]^{(m-k+1)}\bigg|_{x=1^-}\biggr)\\
=2^k\bell_{m,k}\biggl(-\frac{1}{12},\frac{2}{45},-\frac{3}{70}, \frac{32}{525},-\frac{80}{693}, \dotsc, \frac{(2m-2k+2)!!}{(2m-2k+4)!} Q(2,2m-2k+2)\biggr)\\
=(-2)^{k}[2(m-k)]!!\binom{m}{k}\sum_{j=1}^{k}(-1)^{j}(2j)!\binom{k}{j} \frac{Q(2j,2m)}{(2j+2m)!},
\end{gathered}
\end{equation}
where $Q(2j,2m)$ is defined by~\eqref{Q(m-k)-sum-dfn}.
\end{thm}

\begin{proof}
Let
\begin{equation*}
x_m=\biggl[\frac{(\arccos x)^{2}}{2(1-x)}\biggr]^{(m)}\bigg|_{x=1^-}, \quad m\in\mathbb{N}.
\end{equation*}
Then, from~\eqref{113-final-formula} and~\eqref{arccos-2(1-x)-power-ser-expan}, it follows that
\begin{gather*}
\bell_{n+k,k}(x_1,x_2,\dotsc,x_{n+1})
=\binom{n+k}{k}\lim_{t\to0}\frac{\td^{n}}{\td t^{n}}\Biggl[\sum_{m=0}^{\infty} \frac{x_{m+1}}{(m+1)!}t^{m}\Biggr]^k\\
=\binom{n+k}{k}\lim_{t\to0}\frac{\td^{n}}{\td t^{n}}\Biggl(\frac{1}{t}\sum_{m=1}^{\infty} \biggl[\frac{(\arccos x)^{2}}{2(1-x)}\biggr]^{(m)} \bigg|_{x=1^-} \frac{t^{m}}{m!}\Biggr)^k\\
=\binom{n+k}{k}\lim_{x\to1^-}\frac{\td^{n}}{\td x^{n}}\Biggl(\frac{1}{x-1}\sum_{m=1}^{\infty} \biggl[\frac{(\arccos x)^{2}}{2(1-x)}\biggr]^{(m)} \bigg|_{x=1^-} \frac{(x-1)^{m}}{m!}\Biggr)^k\\
=\binom{n+k}{k}\lim_{x\to1^-}\frac{\td^{n}}{\td x^{n}}\biggl(\frac{1}{x-1} \biggl[\frac{(\arccos x)^{2}}{2(1-x)}-1\biggr]\biggr)^k\\
=\binom{n+k}{k}\lim_{x\to1^-}\frac{\td^{n}}{\td x^{n}}\Biggl(\frac{1}{(x-1)^k} \sum_{j=0}^{k}(-1)^{k-j}\binom{k}{j} \biggl[\frac{(\arccos x)^{2}} {2(1-x)}\biggr]^j\Biggr)\\
=\binom{n+k}{k}\lim_{x\to1^-}\frac{\td^{n}}{\td x^{n}}\Biggl[\frac{(-1)^k}{(x-1)^k} +\frac{1}{(x-1)^k}\sum_{j=1}^{k}(-1)^{k-j}\binom{k}{j} \Biggl(1+(2j)!\sum_{m=1}^{\infty} \frac{Q(2j,2m)}{(2j+2m)!}[2(x-1)]^{m}\Biggr)\Biggr]\\
=(-1)^{k}\binom{n+k}{k}\lim_{x\to1^-}\frac{\td^{n}}{\td x^{n}}\Biggl(\sum_{m=1}^{\infty} 2^m\Biggl[\sum_{j=1}^{k}(-1)^{j}(2j)!\binom{k}{j} \frac{Q(2j,2m)}{(2j+2m)!}\Biggr] (x-1)^{m-k}\Biggr)
\end{gather*}
for $k\in\mathbb{N}$. This implies that
\begin{equation}\label{Q(2j-2k)-Comb-ID=0}
\sum_{j=1}^{k}(-1)^{j}(2j)!\binom{k}{j} \frac{Q(2j,2m)}{(2j+2m)!}=0, \quad 1\le m<k.
\end{equation}
Accordingly, we derive
\begin{gather*}
\bell_{n+k,k}(x_1,x_2,\dotsc,x_{n+1})
=(-1)^{k}\binom{n+k}{k}\\
\times\lim_{x\to1^-}\frac{\td^{n}}{\td x^{n}}\Biggl(\sum_{m=k}^{\infty} 2^m\Biggl[\sum_{j=1}^{k}(-1)^{j}(2j)!\binom{k}{j} \frac{Q(2j,2m)}{(2j+2m)!}\Biggr] (x-1)^{m-k}\Biggr)\\
=(-2)^{k}\binom{n+k}{k} \lim_{x\to1^-}\frac{\td^{n}}{\td x^{n}}\Biggl(\sum_{m=0}^{\infty} 2^{m}\Biggl[\sum_{j=1}^{k}(-1)^{j}(2j)!\binom{k}{j} \frac{Q(2j,2m+2k)}{(2j+2m+2k)!}\Biggr] (x-1)^{m}\Biggr)\\
=(-2)^{k}\binom{n+k}{k} \lim_{x\to1^-}\sum_{m=n}^{\infty} 2^{m}\Biggl[\sum_{j=1}^{k}(-1)^{j}(2j)!\binom{k}{j} \frac{Q(2j,2m+2k)}{(2j+2m+2k)!}\Biggr]\langle m\rangle_n(x-1)^{m-n}\\
=(-2)^{k}(2n)!!\binom{n+k}{k}\sum_{j=1}^{k}(-1)^{j}(2j)!\binom{k}{j} \frac{Q(2j,2n+2k)}{(2j+2n+2k)!}
\end{gather*}
for $n\ge k\in\mathbb{N}$. Replacing $n+k$ by $m$ results in
\begin{equation*}
\bell_{m,k}(x_1,x_2,\dotsc,x_{m-k+1})=(-2)^{k}[2(m-k)]!!\binom{m}{k}\sum_{j=1}^{k}(-1)^{j}(2j)!\binom{k}{j} \frac{Q(2j,2m)}{(2j+2m)!}
\end{equation*}
for $m\ge k\in\mathbb{N}$. The required result is thus proved.
\end{proof}

\begin{thm}\label{arccos-Taylor-alpha-THM}
For $\alpha\in\mathbb{R}$, we have
\begin{equation}\label{arccos-Taylor-alpha-Eq}
\biggl[\frac{(\arccos x)^{2}}{2(1-x)}\biggr]^\alpha
=1+\sum_{n=1}^{\infty} \Biggl[\sum_{j=1}^{n}(-1)^{j} \frac{\langle\alpha\rangle_j}{j!} \sum_{\ell=1}^{j}(-1)^{\ell}(2\ell)!\binom{j}{\ell} \frac{Q(2\ell,2n)}{(2\ell+2n)!}\Biggr][2(x-1)]^n.
\end{equation}
\end{thm}

\begin{proof}
By virtue of the Fa\`a di Bruno formula~\eqref{Bruno-Bell-Polynomial} and the formula~\eqref{arccos-squar-(1-x)-bell-Eq} in Theorem~\ref{arccos-squar-(1-x)-bell-thm}, we obtain
\begin{equation*}
\biggl(\biggl[\frac{(\arccos x)^{2}}{2(1-x)}\biggr]^\alpha\biggr)^{(n)}
=\sum_{j=1}^{n}\langle\alpha\rangle_j \biggl[\frac{(\arccos x)^{2}}{2(1-x)}\biggr]^{\alpha-j} \bell_{n,j}\biggl(\biggl[\frac{(\arccos x)^{2}}{2(1-x)}\biggr]', \dotsc, \biggl[\frac{(\arccos x)^{2}}{2(1-x)}\biggr]^{(n-j+1)}\biggr)
\end{equation*}
for $n\in\mathbb{N}$. Taking the limit $x\to1^-$ and employing~\eqref{arccos-squar-(1-x)-bell-Eq} in Theorem~\ref{arccos-squar-(1-x)-bell-thm} lead to
\begin{gather*}
\lim_{x\to1^-}\biggl(\biggl[\frac{(\arccos x)^{2}}{2(1-x)}\biggr]^\alpha\biggr)^{(n)}
=\sum_{j=1}^{n}\langle\alpha\rangle_j \bell_{n,j}\biggl(\lim_{x\to1^-}\biggl[\frac{(\arccos x)^{2}}{2(1-x)}\biggr]', \dotsc, \lim_{x\to1^-}\biggl[\frac{(\arccos x)^{2}}{2(1-x)}\biggr]^{(n-j+1)}\biggr)\\
=\sum_{j=1}^{n}\langle\alpha\rangle_j (-2)^{j}[2(n-j)]!!\binom{n}{j} \sum_{\ell=1}^{j}(-1)^{\ell}(2\ell)!\binom{j}{\ell} \frac{Q(2\ell,2n)}{(2\ell+2n)!}
\end{gather*}
for $n\in\mathbb{N}$. Consequently, the required result~\eqref{arccos-Taylor-alpha-Eq} is proved.
\end{proof}

\begin{cor}\label{Last-v2-comb-id-arcos-Cor}
For $k,n\in\mathbb{N}$, we have
\begin{equation}\label{Last-v2-comb-id-arcos}
\sum_{j=1}^{n}(-1)^{j} \frac{\langle k\rangle_j}{j!} \sum_{\ell=1}^{j}(-1)^{\ell}(2\ell)!\binom{j}{\ell} \frac{Q(2\ell,2n)}{(2\ell+2n)!}
=(2k)!\frac{Q(2k,2n)}{(2k+2n)!}.
\end{equation}
\end{cor}

\begin{proof}
This follows from letting $\alpha=k\in\mathbb{N}$ in~\eqref{arccos-Taylor-alpha-Eq} and equating coefficients of factors $(x-1)^n$ in~\eqref{arccos-2(1-x)-power-ser-expan}.
\end{proof}

\begin{cor}\label{Pi-Taylor-alpha-Cor}
For $\alpha\in\mathbb{R}$, we have
\begin{equation}\label{Pi-Taylor-alpha-Eq}
\biggl(\frac{\pi^2}{9}\biggr)^\alpha
=1+\sum_{n=1}^{\infty} (-1)^n\sum_{j=1}^{n}(-1)^{j} \frac{\langle\alpha\rangle_j}{j!} \sum_{\ell=1}^{j}(-1)^{\ell}(2\ell)!\binom{j}{\ell} \frac{Q(2\ell,2n)}{(2\ell+2n)!}.
\end{equation}
\end{cor}

\begin{proof}
This follows from setting $x=\frac{1}{2}$ in~\eqref{arccos-Taylor-alpha-Eq}.
\end{proof}

\begin{rem}
The formula~\eqref{arccos-squar-(1-x)-bell-Eq} in Theorem~\ref{arccos-squar-(1-x)-bell-thm} can be used to compute Taylor's series expansions of functions like $f\bigl(\frac{(\arccos x)^{2}}{2(1-x)}\bigr)$ around the point $x=1^-$, only if all the derivatives of $f(x)$ at $x=1^-$ are explicitly computable.
\end{rem}

\section{Recovering Maclaurin's series expansion of $\bigl(\frac{\arcsin x}{x}\bigr)^k$}\label{Recovering-Sec-arccos}

In this section, by virtue of some conclusions in Lemmas~\ref{square-sum-prod-lem}, \ref{square-sum-prod-odd-lem}, and~\ref{cosh-sinh-arcsin-arccos-thm}, we recover Maclaurin's series expansions~\eqref{arcsin-series-expansion-unify} and~\eqref{arcsinh-series-expansion} in Theorems~\ref{arcsin-series-expansion-unify-thm} and~\ref{arcsinh-identity-thm}.
\par
It is easy to see that
\begin{equation*}
\cosh t=\frac{\te^t+\te^{-t}}{2}=\sum_{k=0}^{\infty}\frac{t^{2k}}{(2k)!}.
\end{equation*}
Then, making use of the identity~\eqref{square-sum-first-stirl-eq} in Lemma~\ref{square-sum-prod-lem}, the series expansion~\eqref{cosh-arcsin-ser} can be reformulated as
\begin{align*}
\sum_{k=0}^{\infty}\frac{(\arcsin x)^{2k}}{(2k)!} \alpha^{2k}
&=1+\frac{x^2}{2}\alpha^2+\alpha^2\sum_{k=2}^{\infty}4^{k-1} \Biggl(\prod_{\ell=1}^{k-1}\biggl[\ell^2+\biggl(\frac{\alpha}{2}\biggr)^2\biggr]\Biggr) \frac{x^{2k}}{(2k)!}\\
&=1+\frac{x^2}{2}\alpha^2+\sum_{k=2}^{\infty}(-4)^{k-1} \frac{x^{2k}}{(2k)!} \sum_{j=1}^{k} \frac{(-1)^{j-1}}{4^{j-1}}\\
&\quad\times\Biggl[\sum_{\ell=2j-1}^{2k-1}\binom{\ell}{2j-1}s(2k-1,\ell)(k-1)^{\ell-2j+1}\Biggr] \alpha^{2j}\\
&=1+\frac{x^2}{2}\alpha^2+\alpha^{2}\sum_{k=2}^{\infty}(-4)^{k-1} \frac{x^{2k}}{(2k)!} \Biggl[\sum_{\ell=1}^{2k-1}\ell s(2k-1,\ell)(k-1)^{\ell-1}\Biggr] \\
&\quad+\sum_{k=2}^{\infty}\frac{x^{2k}}{(2k)!} \sum_{j=2}^{k} (-4)^{k-j} \Biggl[\sum_{\ell=2j-1}^{2k-1}\binom{\ell}{2j-1}s(2k-1,\ell)(k-1)^{\ell-2j+1}\Biggr] \alpha^{2j}\\
&=1+\frac{x^2}{2}\alpha^2+\alpha^{2}\sum_{k=2}^{\infty}(-4)^{k-1} \frac{x^{2k}}{(2k)!} (-1)^{k-1}[(k-1)!]^2 \\
&\quad+\sum_{j=2}^{\infty} \sum_{k=j}^{\infty} \frac{x^{2k}}{(2k)!} (-4)^{k-j} \Biggl[\sum_{\ell=2j-1}^{2k-1} \binom{\ell}{2j-1}s(2k-1,\ell)(k-1)^{\ell-2j+1}\Biggr] \alpha^{2j}\\
&=1+\frac{x^2}{2}\alpha^2+\alpha^{2}\sum_{k=2}^{\infty}[(2k-2)!!]^2 \frac{x^{2k}}{(2k)!} \\
&\quad+\sum_{k=2}^{\infty} \sum_{m=k}^{\infty} \frac{x^{2m}}{(2m)!} (-4)^{m-k} \Biggl[\sum_{\ell=2k-1}^{2m-1} \binom{\ell}{2k-1}s(2m-1,\ell)(m-1)^{\ell-2k+1}\Biggr] \alpha^{2k},
\end{align*}
where we used the identity~\eqref{square-sum-first-alpha=0} or~\eqref{square-sum-first-alpha=0-Q}. Regarding $\alpha$ as a variable and equating coefficients of $\alpha^{2k}$ arrive at
\begin{equation}\label{arcsin-ser-expan-square}
\frac{(\arcsin x)^{2}}{2}=\frac{x^2}{2}+\sum_{k=2}^{\infty}[(2k-2)!!]^2 \frac{x^{2k}}{(2k)!}
=\frac{1}{2}\sum_{k=1}^{\infty} \frac{(2k-2)!!}{(2k-1)!!}\frac{x^{2k}}{k}
\end{equation}
and
\begin{equation}\label{arcsin-ser-expan-k>4}
\begin{split}
\frac{(\arcsin x)^{2k}}{(2k)!}
&=\sum_{m=k}^{\infty} (-4)^{m-k} \Biggl[\sum_{\ell=2k-1}^{2m-1} \binom{\ell}{2k-1}s(2m-1,\ell)(m-1)^{\ell-2k+1}\Biggr] \frac{x^{2m}}{(2m)!}\\
&=\frac{x^{2k}}{(2k)!}\sum_{m=0}^{\infty} \frac{(-1)^{m}}{\binom{2k+2m}{2k}} Q(2k,2m) \frac{(2x)^{2m}}{(2m)!}
\end{split}
\end{equation}
for $k\ge2$.
\par
Making use of the series expansion~\eqref{sinh-arcsin-ser} and the identity~\eqref{square-sum-odd-id} in Lemma~\ref{square-sum-prod-odd-lem}, we obtain
\begin{align*}
\sum_{k=0}^{\infty}\frac{(\arcsin x)^{2k+1}}{(2k+1)!}\alpha^{2k+1}
&=\alpha\sum_{k=0}^{\infty} (-1)^{k}2^{2k}\frac{x^{2k+1}}{(2k+1)!} \sum_{j=0}^{2k} (-1)^j \Biggl[\sum_{\ell=2j}^{2k}\frac{s(2k,\ell)}{2^\ell}\binom{\ell}{2j} (2k-1)^{\ell-2j}\Biggr]\alpha^{2j}\\
&=\sum_{j=0}^{\infty}(-1)^j \Biggl[\sum_{k=\lceil{j/2}\rceil}^{\infty} (-1)^{k}\frac{(2x)^{2k+1}}{(2k+1)!}  \sum_{\ell=2j}^{2k}\frac{s(2k,\ell)}{2^{\ell+1}}\binom{\ell}{2j} (2k-1)^{\ell-2j}\Biggr]\alpha^{2j+1}\\
&=\sum_{k=0}^{\infty}(-1)^k \Biggl[\sum_{m=\lceil{k/2}\rceil}^{\infty} (-1)^{m}\frac{(2x)^{2m+1}}{(2m+1)!}  \sum_{\ell=2k}^{2m}\frac{s(2m,\ell)}{2^{\ell+1}}\binom{\ell}{2k} (2m-1)^{\ell-2k}\Biggr]\alpha^{2k+1},
\end{align*}
where we used the identity~\eqref{square-sum-odd-alpha=0} or~\eqref{square-sum-odd-alpha=0-Q} and $\lceil x\rceil$ stands for the ceiling function which gives the smallest integer not less than $x$. Regarding $\alpha$ as a variable and equating coefficients of $\alpha^{2k+1}$ reduce to
\begin{equation}\label{arcsin-ser-expan-(2k+1)}
\begin{split}
\frac{(\arcsin x)^{2k+1}}{(2k+1)!}
&=(-1)^k\sum_{m=\lceil{k/2}\rceil}^{\infty} (-1)^{m} \Biggl[\sum_{\ell=2k}^{2m}\frac{s(2m,\ell)}{2^{\ell+1}}\binom{\ell}{2k} (2m-1)^{\ell-2k}\Biggr]\frac{(2x)^{2m+1}}{(2m+1)!}\\
&=\frac{x^{2k+1}}{(2k+1)!}\sum_{m=0}^{\infty} \frac{(-1)^m}{\binom{2k+2m+1}{2k+1}} Q(2k+1,2m) \frac{(2x)^{2m}}{(2m)!}.
\end{split}
\end{equation}
\par
Combining the series expansions~\eqref{arcsin-ser-expan-square}, \eqref{arcsin-ser-expan-k>4}, and~\eqref{arcsin-ser-expan-(2k+1)} leads to the series expansion~\eqref{arcsin-series-expansion-unify}.
\par
By similar arguments as above, from the series expansions~\eqref{cos-arcsin-ser} and~\eqref{sin-arcsin-ser}, we can recover series expansion~\eqref{arcsin-series-expansion-unify} once again.
\par
Utilizing the relation $\arcsinh t=-\ti\arcsin(\ti t)$ or, equivalently, the relation $\arcsin t=-\ti\arcsinh(\ti t)$, the series expansions~\eqref{arcsin-series-expansion-unify} and~\eqref{arcsinh-series-expansion} can be derived from each other.

\section{Conclusions}
In this paper, by virtue of Lemmas~\ref{square-sum-prod-lem} and~\ref{square-sum-prod-odd-lem}, with the aid of Taylor's series expansion~\eqref{cos-arccos-ser} or~\eqref{cosh-arccos-ser} in Lemma~\ref{cosh-sinh-arcsin-arccos-thm}, and in the light of properties recited in Lemma~\ref{Partial-Bell-Conclusions-Lem} of partial Bell polynomials, the author establishes Taylor's series expansions~\eqref{arccos-2(1-x)-power-ser-expan}, \eqref{arccosh-square-ser}, and~\eqref{arccos-Taylor-alpha-Eq} in Theorems~\ref{arccos-ser-expan-thm} and~\ref{arccos-Taylor-alpha-THM}, presents an explicit formula~\eqref{arccos-squar-(1-x)-bell-Eq}, derives several combinatorial identities~\eqref{square-sum-first-alpha=0}, \eqref{Stirl-first-ID}, \eqref{first-stirl-2id}, \eqref{comb-id-arccos-ven=pow}, \eqref{Q(2j-2k)-Comb-ID=0}, and~\eqref{Last-v2-comb-id-arcos}, demonstrates several series representations~\eqref{Pi-power-ser-repres}, \eqref{pi-square-div-8}, and~\eqref{Pi-Taylor-alpha-Eq} in Corollaries~\ref{series-approx-pi-first2cor} and~\ref{Pi-Taylor-alpha-Cor} of the circular constant $\pi$ and its real powers, and recovers Maclaurin's series expansions~\eqref{arcsin-series-expansion-unify} and~\eqref{arcsinh-series-expansion} in Section~\ref{Recovering-Sec-arccos}.
\par
Those conclusions stated in Corollaries~\ref{Minus-sh-square-ser-COR}, \ref{COR-4-Deriv}, \ref{arccosh-squ-deriv-form2-cor}, \ref{arccosH-square-ser-expan-thm}, and~\ref{arccos-odd-power-deriv-thm} are useful, meaningful, and significant.
\par
Several Maclaurin's series expansions of the functions $(\arccos x)^m$ and $(\arccosh x)^m$ for $m\in\mathbb{N}$ have been surveyed, reviewed, collected, and mentioned in~\cite[Section~7]{Ser-Pow-Arcs-Arctan.tex}, but comparatively their forms or formulations are not more beautiful, not more satisfactory, not simpler, not more concise, or not nicer than these newly-established ones in this paper.
\par
This paper is an extended version of the preprint~\cite{maclaurin-arccos.tex} and a companion of the papers~\cite{Ser-Pow-Arcs-Arctan.tex, AIMS-Math20210491.tex, series-arccos-v2.tex, Wilf-Ward-2010P.tex}.

\section{Declarations}

\begin{description}
\item[Acknowledgements]
Not applicable.

\item[Availability of data and material]
Data sharing is not applicable to this article as no new data were created or analyzed in this study.

\item[Competing interests]
The author declares that he has no conflict of competing interests.

\item[Funding]
Not applicable.

\item[Authors' contributions]
Not applicable.

\end{description}


\begin{thebibliography}{99}

\bibitem{abram}
M. Abramowitz and I. A. Stegun (Eds), \textit{Handbook of Mathematical Functions with Formulas, Graphs, and Mathematical Tables}, National Bureau of Standards, Applied Mathematics Series \textbf{55}, 10th printing, Washington, 1972.

\bibitem{Baker-Temme-JMP-1984}
M. Bakker and N. M. Temme, \emph{Sum rule for products of Bessel functions: Comments on a paper by Newberger}, J. Math. Phys. \textbf{25} (1984), no.~5, 1266\nobreakdash--1267; available online at \url{https://doi.org/10.1063/1.526282}.

\bibitem{Baricz-AML-2010}
\'A. Baricz, \emph{Powers of modified Bessel functions of the first kind}, Appl. Math. Lett. \textbf{23} (2010), no.~6, 722\nobreakdash--724; available online at \url{https://doi.org/10.1016/j.aml.2010.02.015}.

\bibitem{Bender-Brody-Meister-JMP-2003}
C. M. Bender, D. C. Brody, and B. K. Meister, \emph{On powers of Bessel functions}, J. Math. Phys. \textbf{44} (2003), no.~1, 309\nobreakdash--314; available online at \url{https://doi.org/10.1063/1.1526940}.

\bibitem{Borwein-Chamberland-IJMMS-2007}
J. M. Borwein and M. Chamberland, \emph{Integer powers of arcsin}, Int. J. Math. Math. Sci. \textbf{2007}, Art. ID~19381, 10~pages; available online at \url{https://doi.org/10.1155/2007/19381}.

\bibitem{Brychkov-ITSF-2009}
Yu. A. Brychkov, \emph{Power expansions of powers of trigonometric functions and series containing Bernoulli and Euler polynomials}, Integral Transforms Spec. Funct. \textbf{20} (2009), no.~11-12, 797\nobreakdash--804; available online at \url{https://doi.org/10.1080/10652460902867718}.

\bibitem{Charalambides-book-2002}
C. A. Charalambides, \textit{Enumerative Combinatorics}, CRC Press Series on Discrete Mathematics and its Applications. Chapman \& Hall/CRC, Boca Raton, FL, 2002.

\bibitem{Comtet-Combinatorics-74}
L. Comtet, \textit{Advanced Combinatorics: The Art of Finite and Infinite Expansions}, Revised and Enlarged Edition, D. Reidel Publishing Co., 1974; available online at \url{https://doi.org/10.1007/978-94-010-2196-8}.

\bibitem{Davydychev-Kalmykov-2001}
A. I. Davydychev and M. Yu. Kalmykov, \emph{New results for the $\varepsilon$-expansion of certain one-, two- and three-loop Feynman diagrams}, Nuclear Phys. B \textbf{605} (2001), no.~1-3, 266\nobreakdash--318; available online at \url{https://doi.org/10.1016/S0550-3213(01)00095-5}.

\bibitem{Ser-Pow-Arcs-Arctan.tex}
B.-N. Guo, D. Lim, and F. Qi, \textit{Maclaurin series expansions for powers of inverse (hyperbolic) sine, for powers of inverse (hyperbolic) tangent, and for incomplete gamma functions, with applications}, arXiv (2021), available online at \url{https://arxiv.org/abs/2101.10686v6}.

\bibitem{AIMS-Math20210491.tex}
B.-N. Guo, D. Lim, and F. Qi, \textit{Series expansions of powers of arcsine, closed forms for special values of Bell polynomials, and series representations of generalized logsine functions}, AIMS Math. \textbf{6} (2021), no.~7, 7494\nobreakdash--7517; available online at \url{https://doi.org/10.3934/math.2021438}.

\bibitem{Hansen-B-1975}
E. R. Hansen, \emph{A Table of Series and Products}, Prentice-Hall, Englewood Cliffs, NJ, USA, 1975.

\bibitem{Bess-Pow-Polyn-CMES.tex}
Y. Hong, B.-N. Guo, and F. Qi, \textit{Determinantal expressions and recursive relations for the Bessel zeta function and for a sequence originating from a series expansion of the power of modified Bessel function of the first kind}, CMES Comput. Model. Eng. Sci. \textbf{129} (2021), no.~1, 409\nobreakdash--423; available online at \url{https://doi.org/10.32604/cmes.2021.016431}.

\bibitem{Howard-Fibonacci-1985}
F. T. Howard, \emph{Integers related to the Bessel function $J_1(z)$}, Fibonacci. Quart. \textbf{23} (1985), no.~3, 249\nobreakdash--257.

\bibitem{Kalmykov-Sheplyakov-lsjk-2005}
M. Yu. Kalmykov and A. Sheplyakov, \emph{lsjk---a C++ library for arbitrary-precision numeric evaluation of the generalized log-sine functions}, Computer Phys. Commun. \textbf{172} (2005), no.~1, 45\nobreakdash--59; available online at \url{https://doi.org/10.1016/j.cpc.2005.04.013}.

\bibitem{Lehmer-Monthly-1985}
D. H. Lehmer, \emph{Interesting series involving the central binomial coefficient}, Amer. Math. Monthly \textbf{92} (1985), no.~7, 449\nobreakdash--457; available online at \url{http://dx.doi.org/10.2307/2322496}.

\bibitem{PAM-Jul-06-2021-0023.tex}
Y.-W. Li and F. Qi, \textit{The sum of an alternating series involving central binomial numbers and its three proofs}, J. Korea Soc. Math. Educ. Ser. B Pure Appl. Math. \textbf{28} (2021), no.~4, in press.

\bibitem{Zeta-luo.tex}
Q.-M. Luo, B.-N. Guo and F. Qi, \textit{On evaluation of Riemann zeta function $\zeta(s)$}, Adv. Stud. Contemp. Math. (Kyungshang) \textbf{7} (2003), no.~2, 135\nobreakdash--144.

\bibitem{Moll-Vignat-IJNT-2014}
V. H. Moll and C. Vignat, \emph{On polynomials connected to powers of Bessel functions}, Int. J. Number Theory \textbf{10} (2014), no.~5, 1245\nobreakdash--1257; available online at \url{https://doi.org/10.1142/S1793042114500249}.

\bibitem{Newberger-Erratum-JMP-1983}
B. S. Newberger, \emph{Erratum: New sum rule for products of Bessel functions with application to plasma physics}, J. Math. Phys. \textbf{24} (1983), no.~8, 2250\nobreakdash--2250; available online at \url{https://doi.org/10.1063/1.525940}.

\bibitem{Newberger-JMP-1982}
B. S. Newberger, \emph{New sum rule for products of Bessel functions with application to plasma physics}, J. Math. Phys. \textbf{23} (1982), no.~7, 1278\nobreakdash--1281; available online at \url{https://doi.org/10.1063/1.525510}.

\bibitem{Oertel-arXiv-2010.00746v2}
F. Oertel, \textit{Grothendieck's inequality and completely correlation preserving functions---a summary of recent results and an indication of related research problems}, arXiv (2020), available online at \url{https://arxiv.org/abs/2010.00746v2}.

\bibitem{Tan-Der-App-Thanks.tex}
F. Qi, \textit{Derivatives of tangent function and tangent numbers}, Appl. Math. Comput. \textbf{268} (2015), 844\nobreakdash--858; available online at \url{http://dx.doi.org/10.1016/j.amc.2015.06.123}.

\bibitem{1st-Stirl-No-adjust.tex}
F. Qi, \textit{Diagonal recurrence relations for the Stirling numbers of the first kind}, Contrib. Discrete Math. \textbf{11} (2016), no.~1, 22\nobreakdash--30; available online at \url{https://doi.org/10.11575/cdm.v11i1.62389}.

\bibitem{series-arccos-v2.tex}
F. Qi, \textit{Explicit formulas for partial Bell polynomials, Maclaurin's series expansions of real powers of inverse (hyperbolic) cosine and sine, and series representations of powers of Pi}, Research Square (2021), available online at \url{https://doi.org/10.21203/rs.3.rs-959177/v3}.

\bibitem{maclaurin-arccos.tex}
F. Qi, \textit{Taylor's series expansions for even powers of inverse cosine function and series representations for powers of Pi}, arXiv (2021), available online at \url{https://arxiv.org/abs/2110.02749v1}.

\bibitem{Qi-Chen-Lim-RNA.tex}
F. Qi, C.-P. Chen, and D. Lim, \textit{Several identities containing central binomial coefficients and derived from series expansions of powers of the arcsine function}, Results Nonlinear Anal. \textbf{4} (2021), no.~1, 57\nobreakdash--64; available online at \url{https://doi.org/10.53006/rna.867047}.

\bibitem{AADM-2821.tex}
F. Qi and B.-N. Guo, \textit{A diagonal recurrence relation for the Stirling numbers of the first kind}, Appl. Anal. Discrete Math. \textbf{12} (2018), no.~1, 153\nobreakdash--165; available online at \url{https://doi.org/10.2298/AADM170405004Q}.

\bibitem{Spec-Bell2Euler-S.tex}
F. Qi and B.-N. Guo, \emph{Explicit formulas for special values of the Bell polynomials of the second kind and for the Euler numbers and polynomials}, Mediterr. J. Math. \textbf{14} (2017), no.~3, Art.~140, 14~pages; available online at \url{https://doi.org/10.1007/s00009-017-0939-1}.

\bibitem{CDM-68111.tex}
F. Qi, D.-W. Niu, D. Lim, and B.-N. Guo, \textit{Closed formulas and identities for the Bell polynomials and falling factorials}, Contrib. Discrete Math. \textbf{15} (2020), no.~1, 163\nobreakdash--174; available online at \url{https://doi.org/10.11575/cdm.v15i1.68111}.

\bibitem{Bell-value-elem-funct.tex}
F. Qi, D.-W. Niu, D. Lim, and Y.-H. Yao, \textit{Special values of the Bell polynomials of the second kind for some sequences and functions}, J. Math. Anal. Appl. \textbf{491} (2020), no.~2, Article 124382, 31~pages; available online at \url{https://doi.org/10.1016/j.jmaa.2020.124382}.

\bibitem{AJOM-D-16-00138.tex}
F. Qi, X.-T. Shi, and F.-F. Liu, \textit{Expansions of the exponential and the logarithm of power series and applications}, Arab. J. Math. (Springer) \textbf{6} (2017), no.~2, 95\nobreakdash--108; available online at \url{https://doi.org/10.1007/s40065-017-0166-4}.

\bibitem{JAAC961.tex}
F. Qi, X.-T. Shi, F.-F. Liu, and D. V. Kruchinin, \textit{Several formulas for special values of the Bell polynomials of the second kind and applications}, J. Appl. Anal. Comput. \textbf{7} (2017), no.~3, 857\nobreakdash--871; available online at \url{https://doi.org/10.11948/2017054}.

\bibitem{Wilf-Ward-2010P.tex}
F. Qi and M. D. Ward, \textit{Closed-form formulas and properties of coefficients in Maclaurin's series expansion of Wilf's function}, arXiv (2021), available online at \url{https://arxiv.org/abs/2110.08576v1}.

\bibitem{Quaintance-Gould-Stirling-B}
J. Quaintance and H. W. Gould, \emph{Combinatorial Identities for Stirling Numbers}. The unpublished notes of H. W. Gould. With a foreword by George E. Andrews. World Scientific Publishing Co. Pte. Ltd., Singapore, 2016.

\bibitem{Qureshi-Majid-Bhat-Aeajmms-2020}
M. I. Qureshi, J. Majid, and A. H. Bhat, \emph{Hypergeometric forms of some composite functions containing $\textup{arccosine}(x)$ using Maclaurin's expansion}, South East Asian J. Math. Math. Sci. \textbf{16} (2020), no.~3, 83\nobreakdash--95.

\bibitem{Temme-96-book}
N. M. Temme, \emph{Special Functions: An Introduction to Classical Functions of Mathematical Physics}, A Wiley-Interscience Publication, John Wiley \& Sons, Inc., New York, 1996; available online at \url{http://dx.doi.org/10.1002/9781118032572}.

\bibitem{Thir-Nanj-1951-India}
V. R. Thiruvenkatachar and T. S. Nanjundiah, \emph{Inequalities concerning Bessel functions and orthogonal polynomials}, Proc. Ind. Acad. Sci. Sect. A \textbf{33} (1951), 373\nobreakdash--384.

\bibitem{DIGCBC-Wei.tex}
C.-F. Wei, \textit{Integral representations and inequalities of extended central binomial coefficients}, Authorea (2021), available online at \url{https://doi.org/10.22541/au.163355849.99215800/v1}.

\bibitem{Yang-Zhen-MIA-2018-Bessel}
Z.-H. Yang and S.-Z. Zheng, \emph{Monotonicity and convexity of the ratios of the first kind modified Bessel functions and applications}, Math. Inequal. Appl. \textbf{21} (2018), no.~1, 107\nobreakdash--125; available online at \url{https://doi.org/10.7153/mia-2018-21-09}.

\end{thebibliography}
\end{document}